\documentclass[a4paper,11pt]{amsart}
\usepackage{amssymb}
\usepackage{amsmath}
\usepackage{pgf,tikz}
\usepackage{mathrsfs}
\usepackage{pgf,xcolor}
\usetikzlibrary{arrows}
\usepackage{a4wide}
 \DeclareMathOperator{\tr}{tr}

\newcommand{\e}{\mathrm{e}}
 
\DeclareMathOperator{\Img}{Im}
\DeclareMathOperator{\Span}{span}
\DeclareMathOperator{\Hom}{Hom}
\DeclareMathOperator{\U}{U} 
\DeclareMathOperator{\rank}{rank}

\newcommand{\CP}{{\mathbb C}P}
\newcommand{\C}{{\mathbb C}}

\newtheorem{theorem}{Theorem}[section]

\newtheorem{lemma}[theorem]{Lemma}

\newtheorem{definition}[theorem]{Definition}
\newtheorem{proposition/definition}[theorem]{Proposition/Definition}

\theoremstyle{definition}

\newtheorem{remark}[theorem]{Remark}

\usepackage{pdfpages}
\usepackage[all]{xy}
\usepackage{todonotes}
\usepackage{hyperref}
\usepackage{tikz-cd}
\usepackage{tikz}

\title[Classification of primitive immersions of constant curvature]
{Classification of primitive immmersions of constant curvature into flag manifolds}
\author[R. Pacheco]{Rui Pacheco}
\address{Centro de Matem\'{a}tica e Aplica\c{c}{\~{o}}es (CMA-UBI), Universidade da Beira Interior, 6201 -- 001
	Covilh{\~{a}}, Portugal.}
\email{rpacheco@ubi.pt, mehmood.ur.rehman@ubi.pt}

\author[M. U. Rehman]{Mehmood-Ur-Rehman}

\thanks{The first author was partially supported by Funda\c{c}\~{a}o para a Ci\^{e}ncia e Tecnologia through the project UID/MAT/00212/2019. The second author was partially supported by Funda\c{c}\~{a}o para a Ci\^{e}ncia e Tecnologia through the grant UI/BD/153058/2022.}


\begin{document}

\maketitle

\begin{abstract}
We classify primitive minimal immersions of constant curvature from the two-sphere $S^2$ into the  low-dimensional flag manifolds $F_{2,1,1}$ and $F_{2,2,1}$. 
\end{abstract}

\section{Introduction}
It is well known result due to E. Calabi  which states that any complex submanifold with constant holomorphic sectional curvature of the complex projective space $\mathbb{C}P^{n}$ is unitarily congruent to a piece of a \emph{round Veronese embedding}  (see \cite{LawsonCalabi} and the references therein). 
In 1988, J. Bolton et al. \cite{Bolton} extended this result  by proving that any conformal minimal immersion of constant curvature from the two-sphere $S^2$ into a complex projective space, up to unitary congruence, belongs to a \emph{Veronese sequence}.
As a generalization of the $\CP^n$ case, the most natural is to consider the same problem for  complex Grassmannians $G_k(\C^n)$. Several researchers have since contributed to this classification problem.    Z. Li and Z.-H. Yu \cite{Li99} provided a complete classification  of minimal two-sphere of constant curvature into $G_2(\C^4)$. Later, X. Jiao and his collaborators classified  various types of constant curvature minimal two-spheres in $G_2(\C^5)$, including nonsingular holomorphic curves, totally unramified holomorphic curves, and those which are neither holomorphic nor antiholomorphic \cite{HeJiao2014,Jiao2004,Jiao2011,JiaoXu2018}. More recently, in \cite{ChiXieXU} the authors constructed explicit examples of constantly curved holomorphic two-spheres of degree 6 in $G_2(\C^5)$, but a complete classification seems still missing.



For isometric immersions, recall that minimality  is equivalent to harmonicity  \cite{eells-sampson,urakawa}, hence the rich theory of harmonic maps of Riemann surfaces  come into play. 
We are particulary interested in certain twistorial constructions of harmonic maps  from Riemann surfaces into symmetric spaces  \cite{burstall,burstallrawnsley}.  An important class of twistor lifts   is that of \emph{primitive maps} into $k$-symmetric spaces $G/K$ \cite{burstall,Guest}. For $k>2$, primitive maps are harmonic with respect to all $G$-invariant metrics.
Thus, a natural problem is to classify constantly curved primitive (minimal) immersions of surfaces into $k$-symmetric spaces, when these are  equipped with  $G$-invariant metrics. 
In \cite{PM} we addressed this problem in the case of  primitive maps $\Psi=(\psi_0,\ldots, \psi_p)$ into the flag manifold 
\begin{equation}\label{homspace}
F_{k_0,\ldots,k_p}=\frac{\U(n)}{\U(k_0)\times\ldots \times \U(k_p)}
\end{equation}
arising from the harmonic sequences $\psi_0,\ldots,\psi_p$   associated to holomorphic maps $\psi_0$ from $S^2$ into the  complex Grassmannian $G_{k_0}(\C^n)$. As a corollary to our results, we showed there that \emph{any full primitive immersion from $S^2$ into the full flag manifold $F_{1,\ldots,1}$ which has constant curvature with respect to at least one  $\U(n)$-invariant metric is unitarily equivalent to the primitive lift of a Veronese map.} 

In the present paper we give the complete classification of primitive (minimal) immersions of constant curvature from $S^2$ into the low-dimensional flags $F_{2,1,1}$ (Section \ref {F211}) and $F_{2,2,1}$ (Section \ref {F221}).    Our methods involve harmonic sequences and harmonic diagrams, as developed in \cite{burstall-wood,PW1}, and a generalization of the notion of absolute value type function \cite{Eschenburg,Li99,JiaoXu2018}, which we called \emph{generalized absolute value type function}.

\section{Preliminaries}

	\subsection{Harmonic sequences and primitive immersions}
		We consider on $\C^{n}$ the standard Hermitian inner product  
$\left<  v, w \right>=v_1\overline{w}_1+\ldots+v_n\overline{w}_n.$
The Grassmannian $G_k(\C^{n})$ of all $k$-dimensional complex subspaces of $\C^{n}$ is a Hermitian symmetric $\U(n)$-space,  with stabilizers conjugate to $\U(k)\times \U(n-k)$.   
The complex structure is given by  
\begin{equation}\label{eq:complexstructure}
	T_L^\C G_k(\C^n)=T_L^{1,0}G_k(\C^n)\oplus T_L^{0,1}G_k(\C^n) \cong \Hom(L,L^\perp)\oplus \Hom(L^\perp,L),
\end{equation}
at $L\in G_k(\C^n)$,
while the compatible Riemann metric is given by
\begin{equation}\label{hmetric}
	h(\xi,\eta)=\frac12\tr (\xi \eta),\qquad \xi,\eta\in T_L  G_k(\C^n).
\end{equation} 
Given a  Riemann surface $M$, we interpret  any smooth map $\psi:M \to G_k(\C^n)$  as a (complex) rank-$k$ vector subbundle, also denoted by $\psi$,  of the trivial vector  bundle $M\times \C^{n}$, with fibre at $z \in M$ given by $\psi(z)$. For the theory of harmonic maps from surfaces into  complex Grassmanians, we refer the reader to \cite{burstall-wood,PW1}.

Let $\psi:M\to G_k(\C^n)$ be a harmonic map and $\{\psi_j\}_{j\in\mathbb{Z}}$ be its harmonic sequence, where: $\psi_0=\psi$; for $j>0$, $\psi_j$ is the $j$-th $\partial'$-Gauss bundle of $\psi$;  for $j<0$, $\psi_j$ is the $j$-th $\partial''$-Gauss bundle of $\psi$ . Recall that, if $\psi$ is holomorphic, then its harmonic sequence must take the form
$$\{0\}=\psi_{-1},\psi_0,\psi_1,\dots,\psi_p,\psi_{p+1}=\{0\},\quad \mbox{with $\psi_i\perp \psi_j$ for all $i\neq j\in\{0,\ldots, p\}$}.$$
 
 Take a local complex coordinate $z$ on $U\subset M$. According to  \eqref{eq:complexstructure}, we have, for each $j\in\mathbb{Z}$,
$$d\psi_j\big(\tfrac{\partial}{\partial z}\big)^{1,0}=A'_{\psi_j},\quad d\psi_j\big(\tfrac{\partial}{\partial z}\big)^{0,1}=-A'_{\psi_{j-1}},$$
where $A'_{\psi_j}$ stands for the $\partial'$-second fundamental form of $\psi_j$.  Assume that $\psi$ has isotropy order $r\geq 2$. Then it is clear that $A'_{\psi_j^\perp}\circ A'_{\psi_j}=0$ for all $j$; hence, each $\psi_j$ is  a branched  conformal minimal immersion.  Set 
\begin{equation}\label{gammaj}
	\gamma_j= \tr  A'_{\psi_j}(A'_{\psi_j})^*.
\end{equation} 
The  metric induced from \eqref{hmetric} by  $\psi_j$ on $M$  is  locally given by 
\begin{equation}\label{metricgamma}
	\psi_j^*h=(\gamma_{j-1}+\gamma_j) dzd\bar z.
\end{equation}
Then, the Laplacian $\Delta_j$ of the metric $\psi_j^*h$ and the  corresponding curvature  are given by
\begin{equation}\label{curvaturepsi}
\Delta_j=\frac{4}{\gamma_{j-1}+\gamma_j}\partial_z \partial_{\bar{z}},\quad 	K(\psi_j)=-\frac{2}{\gamma_{j-1}+\gamma_j} \partial_z \partial_{\bar{z}}\log (\gamma_{j-1}+\gamma_j).
\end{equation}
Here and henceforth we denote $\partial_z:=\frac{\partial}{\partial z}=\frac12(\frac{\partial}{\partial x}-i \frac{\partial}{\partial y})$ and $ \partial_{\bar z}:=\frac{\partial}{\partial\bar{z}}=\frac12(\frac{\partial}{\partial x}+i \frac{\partial}{\partial y}).$

For a holomorphic immersion $\sigma:M\to \CP^n$, the induced metric is locally given by $\sigma^*h=\gamma dzd\bar z$, with
\begin{equation}\label{metricholocp}
	\gamma=\partial_z\partial_{\bar z}\log \|\hat\sigma \|^2,
	\end{equation}
where $\hat\sigma:U\subset M\to \C^{n+1}$ is a local nonvanishing holomorphic section of $\sigma$.  
\subsection{The Veronese sequence}
Recall that the $n$-\emph{Veronese map} $V^n:S^2\to \C P^n$ is the linearly full holomorphic immersion of constant curvature given by
$$V^n(z) =\left[1, \sqrt{n \choose 1}  z, \ldots, \sqrt{n \choose r} z^r, \ldots, z^n\right],
	$$in terms of the canonical complex coordinate $z$ on $\C =S^2\setminus\{\infty\}.$
Up to unitary congruence, the $n$-Veronese map is the unique such immersion, as shown by E. Calabi (see \cite{LawsonCalabi} and the references therein).  The harmonic sequence $V_0^n, \ldots, V_n^n$  associated to $V^n=V^n_0$ is called the \emph{$n$-Veronese sequence}.

\begin{theorem}\label{veronesebolton} \cite[Theorem 5.2]{Bolton}
	The Veronese sequence $V_0^n,\ldots, V^n_n :S^2 \rightarrow \mathbb{C}P^n$ is given by
	$V_j^n=\big[\hat V_j^n \big]$, with $\hat V_j^n=( f_{j,0},\ldots, f_{j,n}),$ where
	\begin{equation}\label{veronesesection}
		f_{j,r}(z)= \frac{j!}{(1+z\bar{z})^j}  \sqrt{n \choose r}  z^{r-j} \sum_{k} (-1)^k   {r \choose j-k}  {n-r \choose k} 
		(z\bar{z})^k.
	\end{equation}
	Moreover, each $V_j^n$ is a minimal immersion with induced metric
	\begin{equation}\label{metric for veronese}
		{V_j^n}^*h=(\gamma_{j-1}+\gamma_j)dzd\bar z=\frac{n+2j(n-j)}{(1+z\bar{z})^2}dzd\bar{z},\quad \gamma^n_j=\frac{(j+1)(n-j)}{(1+z\bar z)^2},
	\end{equation} 
	and hence it has  constant curvature $K(V_j^n)= \frac{4}{n+2j(n-j)}$. 
\end{theorem}
From \eqref{veronesesection}, one obtains
$\|\hat V_j^n\|^2=\frac{n!j!}{(n-j)!}\left(1+z\bar{z}\right)^{n-2j}$;
hence (see \cite{Bolton})
\begin{equation}\label{sectionfarmoulaveronese}
	\hat V_{j+1}^n= \partial_z{\hat V_{j}^n}- \frac{\big\langle \partial_z \hat V_{j}^n,\hat V_{j}^n\big\rangle}{\big\|\hat V_{j}^n\big\|^2}\hat V_{j}^n= \partial_z\hat V_{j}^n-\frac{(n-2j)\bar z}{(1+z\bar z)}\hat V_{j}^n,\quad \mbox{for $0 \le j \le n$}.
\end{equation}
	\subsection{Primitive harmonic immersions } \label{sec:intro}
	 Let $F_{k_0,\ldots,k_p }$ be the flag manifold consisting of all $(p+1)$-tuples  $\Psi=(\psi_0,\dots, \psi_p)$ of mutually orthogonal complex subspaces of $\C^{n}$, so that $\C^n=\psi_0\oplus\ldots \oplus \psi_p$ and $k_j=\dim \psi_j$, for each $j\in\{0,\dots,p\}$. 
	As a homogeneous space,
this flag manifold is given by \eqref{homspace},	where the unitary group $\U(n)$ acts on $F_{k_0,\ldots,k_p }$ as  $g\Psi=(g\psi_0,\dots g\psi_p)$ for  $g\in \U(n)$.  
	We have the following identification:
	\begin{equation}\label{eq:TFdecomposition}
	T^\C_\Psi F_{k_0,\ldots,k_p }\cong \sum_{i\neq j}\Hom(\psi_i, \psi_{j}).
	\end{equation}
	We distinguish the subbundle $T^1$ of $T^\C F_{k_0,\ldots,k_p }$ whose fiber at $\Psi$ is given by 
	$$T^1_\Psi\cong \sum_{j\in\mathbb{Z}_{p+1}} \Hom(\psi_j, \psi_{j+1}).$$

A smooth map $\Psi=(\psi_0,\ldots,\psi_p)$ from a Riemann surface $M$ to $F_{k_0,\ldots,k_p }$
 is said to be \emph{primitive} \cite{burstall} if $d\Psi \big(\tfrac{\partial}{\partial z}  \big)$ is a local section of $\Psi^*T^1$, for all local complex coordinate $z$ on $M$. 
Taking $p=1$, then  $F_{k_0,k_1 }\cong G_{k_0}(\C^n)$, with $k_1=n-k_0$, and all smooth maps into $F_{k_0,k_1}$ are primitive. For $p>1$, if $\Psi=(\psi_0,\ldots,\psi_p)$ is primitive, then each $\psi_j:M\to G_{k_j}(\C^n)$ is harmonic \cite{burstall};  moreover, it is known \cite{black} that, for $p>1$, any primitive map into $F_{k_0,\ldots,k_p }$ is harmonic with respect to all $\U(n)$-invariant metrics of $F_{k_0,\ldots,k_p }$. Any such metric $g$ has the following form \cite{Andreas Arvanitoyeorgos}: given $\xi,\eta\in T_\Psi F_{k_0,\ldots,k_p }$, and writing $\xi=\sum\xi_{ij}$, $\eta=\sum\eta_{ij}$ according to \eqref{eq:TFdecomposition}, then 
\begin{equation}\label{invariantmetric}
g(\xi,\eta)=\sum_{i\neq j} \lambda_{ij}\tr \xi_{ij}\eta_{ij}^*,
\end{equation}
for some positive constants $\lambda_{ij}$ satisfying $\lambda_{ij}=\lambda_{ji}$.  Hence, while the Grassmannians admit a unique (up to multiplication by a positive constant) $\U(n)$-invariant metric, the flag $F_{k_0,\ldots,k_p }$ admits infinite nonequivalent $\U(n)$-invariant metrics. 

Henceforth we will assume $p>1$. Given a primitive immersion $\Psi=(\psi_0,\ldots,\psi_p):M\to F_{k_0,\ldots,k_p}$ and an invariant metric $g$ on  $F_{k_0,\ldots,k_p }$ of the form \eqref{invariantmetric}, then the metric on $M$ induced by $\Psi$ is given, in terms of a local complex coordinate $z$, by
\begin{equation}\label{eq:metric1}
	\Psi^*g=\sum_{j\in \mathbb{Z}_{p+1}}\lambda_j\gamma_jdzd\bar z,
\end{equation}
where, for each $j\in\mathbb{Z}_{p+1}$,  we denote $\lambda_j=\lambda_{j,j+1}$;  and $\gamma_j$ is given by \eqref{gammaj}. 
The corresponding curvature is given by
\begin{equation}\label{kpsi}
	K(\Psi)=-\frac{2}{\sum_{j\in\mathbb{Z}_{p+1}}  \lambda_j \gamma_j } 
	\partial_{z}{\partial_{\bar{z}}}\log \sum_{{j\in\mathbb{Z}_{p+1}}}  \lambda_j \gamma_j.
\end{equation}

 Let $\psi: M \to G_{k}(\C^n)$ be a harmonic map of isotropy order $\geq r$, and $\{\psi_j\}_{j\in \mathbb{Z}}$ be the harmonic sequence of $\psi$. Then the smooth map $\Psi=(\psi_0,\ldots,\psi_{r-1},R)$ into $F_{k_0,\ldots,k_{r-1}, k_R}$,
	with $R=\big(\bigoplus_{j=0}^{r-1}\psi_j\big)^\perp$, $k_j=\mathrm{rank}\,\psi_j$ and $ k_R=\mathrm{rank}\,R$,
	is primitive. 	We say that $\Psi$ is a \emph{primitive lift} of $\psi$. For example, the smooth map $\mathcal{V}^n=(V_0^n,\ldots , V_n^n)$ is a primitive immersion into the full flag manifold $F_{1,\ldots ,1}$. We call  $\mathcal{V}^n$ the \emph{$n$-Veronese primitive immersion.}

	 \subsection{Homogeneous projections and adding constants}\label{homadd}
	 Let $\psi_{0}:M \to G_{k_{0}}(\C^n)$ be a linearly full holomorphic map.  Consider its  harmonic sequence 
	 $\psi_{0},\ldots, \psi_p$ and the corresponding primitive lift
	 $\Psi=(\psi_{0},\ldots, \psi_p)$ into $F_{k_0,\ldots,k_p}.$
	 For a given choice of numbers $0\leq q_0<\ldots <q_{l-1} <q_{l}=p,$
	 define the \emph{homogeneous projection} 
	 $\Pi: F_{k_0,\ldots,k_p} \to F_{k^0,\ldots,k^l},$
	 by
	 $\Pi(\psi_{0},\ldots, \psi_p)= (\psi^0,\ldots, \psi^l),$ where, for each $j\in\{0,\ldots, l\}$,
	 setting $q_{-1}=-1$, 
	 \begin{equation}\label{psi^j}k^j=k_{q_{j-1}+1}+\ldots +  k_{q_{j}},\quad \psi^{j}= \psi_{q_{j-1}+1}\oplus\ldots\oplus\psi_{q_{j}}.\end{equation}
	 The following provides a generalization of \cite[Theorem 2.1]{PM}:

	 \begin{theorem} Let $\psi_0:M \to G_{k_0}(\C^n)$  be a linearly full holomorhic map, with harmonic sequence $\psi_0,\ldots, \psi_p$. 
	 	If the primitive lift $\Psi=(\psi_0,\ldots, \psi_p):M\to F_{k_0,\ldots,k_p }$ is an immersion and there exists at least one $\U(n)$-invariant metric on $F_{k_0,\ldots,k_p }$ with respect to which  $\Psi$ has constant curvature, then $\tilde \Psi= \Pi \circ \Psi:M \to F_{k^0,\ldots,k^l}$ is a  primitive immersion of constant curvature with respect to all   $\U(n)$-invariant metrics on $F_{k^0,\ldots,k^l}$. 
	 \end{theorem}
	 \begin{proof}
	 	For $\psi^j$ defined by \eqref{psi^j}, observe that the second fundamental form  $A'_{\psi^j}:\psi^j\to (\psi^j)^\perp$ is given by: 
	 $A'_{\psi^j}(v)=A'_{\psi_{q_{j}}}\circ \pi_{q_{j}}(v),$
	 where $\pi_{q_{j}}$ is the Hermitian projection onto $\psi_{q_j}$. 
	 Hence the image of  $A'_{\psi^j}$ is contained in $\psi_{q_j+1}\subset \psi^{j+1}$, which means that $\tilde \Psi= \Pi \circ \Psi$  defines a primitive map into $F_{k^0,\ldots,k^l}$. Moreover, in view of \eqref{gammaj}, we have
	 \begin{equation}\label{gamma^j}
	 	\tr  A'_{\psi^j}\left(A'_{\psi^j}\right)^*=\tr  A'_{\psi_{q_j}}\big(A'_{\psi_{q_j}}\big)^*=\gamma_{q_j}.\end{equation}

	 	By \cite[Theorem 2.1]{PM}, each $\psi_j$, with  $j \in \{0,\ldots,p-1\}$ is an immersion of constant curvature and constant K\"{a}hler angle.  This implies that
	 	each function $\gamma_{j}$, on a local chart $(U,z)$, takes the form
	 	$\gamma_{j}= \frac{\alpha_j}{(1+z\bar z)^2}$,
	 	for some positive constant $\alpha_j$.
	 	Now, take any $\U(n)$-invariant metric $g$ on $F_{k^0,\ldots,k^l}$. By \eqref{eq:metric1} and \eqref{gamma^j}, on a local complex chart $(U,z)$,
	 		$\tilde\Psi^*g=\sum_{j=0}^{l-1}\lambda_j \gamma_{q_j}dzd\bar z$,
	for some positive numbers $\lambda_j$, with $j\in\{0,\ldots,l-1\}$. 
	 	Then
	 	$$\tilde\Psi^*g= \sum_{j=0}^{l-1}\frac{\lambda_{j} \alpha_{q_j}}{(1+z\bar z)^2}dzd\bar z,$$
	 	which implies that $\tilde \Psi$ has constant curvature.
	 \end{proof}
	 
Given a primitive map $\Psi=(\psi_{0},\ldots, \psi_p)$  into $F_{k_0,\ldots ,k_p}$,	 
we describe three other processes of 
  obtaining new primitive maps from $\Psi$:
\begin{enumerate} \item \emph{Adding a constant.} The smooth map $\Psi\oplus_r \C^m=(\psi_{0},\ldots,\psi_r\oplus \C^m,\ldots, \psi_p)$ defines a primitive map into $F_{k_0,\ldots,k_r+m,\ldots ,k_p}$.
	 \item \emph{Adding a primitive map.} More generally, let $\Phi=(\varphi_{0},\ldots, \varphi_p)$ be another  primitive map into $F_{s_0,\ldots ,s_p}$. Then $\Psi\oplus \Phi= (\psi_0\oplus\varphi_{0},\ldots, \psi_p\oplus\varphi_p)$ is  primitive map into $F_{k_0+s_0,\ldots ,k_p+s_p}$.
	 \item \emph{Shifting.} $ \tau\Psi=(\psi_{p},\psi_0,\ldots, \psi_{p-1})$ defines a primitive map into $F_{k_p,k_0,\ldots ,k_{p-1}}$. 
\end{enumerate}
	 \subsection{Generealized absolute valued type functions}
	Recall  \cite{Eschenburg,Li99,JiaoXu2018} that  a nonnegative smooth function $f$ on a Riemann surface $M$ is said to be of \emph{absolute value type}, \cite{Eschenburg,Li99} if either $f$ is identically zero or, for any point $x \in M$, there is a local complex chart $(U,z)$ centered at $x$ such that 	$f(z)= (z\bar{z})^p f_1(z)$, where $p\geq 0$ is an integer  
	 and $f_1$ is a nonvanishing smooth function on $U$. In \cite{JiaoXu2018}, the authors admit $p$ to be any positive or negative integer. The following generalization will play an important role in our study of primitive immersions of constant curvature into flag manifolds.

	 \begin{definition}
	 	A  continuous function $f:M\to [0,\infty]$ on a Riemann surface $M$ is said to be of \emph{generalized absolute value type function} if either $f$ is identically zero or,  for any point $x \in M$, there is a local complex chart $(U,z)$ centered at $x$ (i.e. $z(x)=0$), such that 
$f(z)= (z\bar{z})^p f_1(z)$,
	 	for some $ p\in \mathbb{R}$ and positive smooth function $f_1:U\to (0,\infty)$.  If $p>0$, we say that $x$ is a \emph{zero of order} $p$; if $p<0$, we say that $x$ is a \emph{pole of order} $-p$. We denote $p=\mathrm{order}_x(f)$. 
	 \end{definition}
		From the definition, it follows that the zeros and poles of a generalized absolute value type function $f$ are isolated. Moreover, if $M$ is compact, the set	$S$ of zeros and poles is finite; in this case, we set 
	$\mathrm{order}(f)=\sum_{x\in S} \mathrm{order}_x(f).$

 The following lemmas  extend \cite[Lemma 4.1]{Eschenburg},  \cite[Lemma 2.3]{Li99} and  \cite[Lemma 2.8]{JiaoXu2018} to generalized absolute value type functions. We will omit the proofs, since these are straightforward adaptations of  the ones on absolute value type functions. 
	 \begin{lemma}
	 	Let $f$ be a generalized absolute value type function on a compact Riemann surface $M$, equipped with a conformal Riemannian metric. Then  \begin{equation*}\label{integralLaplacian}
	 		\int_{M} \Delta_M \log f\,dA=-4\pi \, \mathrm{order}(f).
	 	\end{equation*}
	 \end{lemma}

	 \begin{lemma}\label{Absolute function lemma2}
	 	Consider on $S^2$ a conformal Riemannian metric of contant curvature $K$. Let $f$ be a generalized absolute value type function on $S^2$. Assume that
	 	there exists a constant $c$ such that
	 		$\frac{1}{4}\Delta_{S^2}\log f=c$,
		off the set of zeros and poles. Then the following holds: (1)
 the constant $c$ is given by
	 	$c=-\frac{K}{4} \mathrm{order}(f);$ (2)
	 	 in terms of the canonical complex coordinate $z$ on $\C=S^2\setminus\{\infty\}$, we have $$f(z)= a\frac{\left|z-z_0\right|^{2p_0}\ldots \left|z-z_m\right|^{2p_{m}}}{\left|z-w_0\right|^{2q_0}\ldots \left|z-w_l\right|^{2q_{l}}}    \left(1+z\bar{z}\right)^{-\mathrm{order}(f)},$$  
	 		where:  $z_0,\ldots, z_m$ are the zeros of $f$ contained in  $\mathbb{C}$, and  $p_0,\ldots , p_m$ are their real orders; $w_0,\ldots, w_m$ are the poles of $f$ contained in  $\mathbb{C}$, and  $q_0,\ldots , q_l$ are  their real orders; $a$ is a positive real constant.
	 \end{lemma}

	 \section{Primitive immersions into $F_{2,1,1}$}\label{F211}
	 
	 \begin{theorem}\label{mainresult1}
	 	Let $\psi_0:S^2 \to G_{2}(\C^4)$ be a  linearly full harmonic map of isotropy order $r\geq 2$, with 
	 	first and second $\partial '$-Gauss bundles $\psi_1$ and $\psi_2$, respectively.
	 	Assume that  the primitive map $\Psi= \left(\psi_0,\psi_1, \psi_2\right):S^2 \to F_{2,1,1}$ is an immersion. If there exists at least one $\U(4)$-invariant metric on $F_{2,1,1}$ with respect to which $\Psi$ has constant curvature, then $\Psi$ is unitarily congruent to one of the following curves:
	 	\begin{enumerate}
	 		\item  $\Psi_1=  ( V_{0}^{3} \oplus V_{1}^{3},   V_{2}^{3},  V_{3}^{3})$;
	 		\item $\Psi_2= ( V_{0}^{2} \oplus \C, V_{1}^{2}, V_{2}^{2})$ (taking  $\C^4=\C^3\oplus \C$);
	 		\item $\Psi_2=  ( V_{0}^{3}\oplus V_{3}^{3}, V_{1}^{3}, V_{2}^{3})$.

	 	\end{enumerate}
	 	For a $\U(4)$-invariant metric on $F_{2,1,1}$ with parameters $\lambda_0,\lambda_1,\lambda_2$, using the notations of \eqref{eq:metric1}, the  metrics  induced by $\Psi_1$, $\Psi_2$ and $\Psi_3$, respectively, have curvatures
	 	$K(\Psi_1)=\frac{4}{4 \lambda_{0}+ 3 \lambda_{1}}$, $	K(\Psi_2) = \frac{2}{\lambda_{0}+ \lambda_1}$, and $  K(\Psi_3) =\frac{4}{3 \lambda_{0}+ 4 \lambda_{1}+ 3\lambda_{2}}$.  
	 \end{theorem}
	 \begin{proof}
 By dimension constraints, either  $r=\infty$  or $2$.	 	
	 	
	 	Suppose that $\psi_0$ is holomorphic, i.e., $r=\infty$.  Since $\Psi$ has constant curvature with respect to at least one $\U(4)$-invariant metric on $F_{2,1,1}$, then  $\Psi$ has constant curvature  with respect to all   $\U(4)$-invariant metrics on $F_{2,1,1}$ and each $\psi_{i}$ is an immersion of constant curvature (see \cite[Theorem 2.1]{PM}). The  fundamental harmonic diagram for this case is the following:
	 	\begin{equation*}\label{second harmonic sequence}
	 		\begin{tikzpicture}
	 		\node (beta0) at (0, 0) {$\psi_0$};
	 		\node (beta1) at (1.1, 0) {$\psi_1$};
	 		\node (beta2) at (2.2, 0) {$\psi_2$};
	 		\draw[->] (beta0) -- (beta1);
	 		\draw[->] (beta1) -- (beta2);
	 	\end{tikzpicture}.
	 	\end{equation*}
	 	with $\rank \psi_0=2$ and $\rank \psi_1=\rank\psi_2=1$. For definitions and properties of harmonic diagrams, see \cite{burstall-wood,PW1}. 
	 	Set $f_0= \underline{\ker} A'_{\psi_0}$ and $f_{1}= f_0^\perp \cap \psi_{0}$. Then $f_0= \underline{\ker} A'_{\psi_0}$ is a holomorphic subbundle of $\psi_{0}$ \cite[Proposition 2.2]{burstall-wood}. Thus,  the fundamental diagram admits the following refinement: 
	 	\begin{equation}\label{diag:G10}
	 		\begin{tikzpicture}
	 			\node (beta0) at (0, 1.1) {$f_1$};
	 			\node (beta1) at (1.1, 1.1) {$f_2$};
	 			\node (beta2) at (2.2, 1.1) {$f_3$};
	 			\node (alpha0) at (0, 0) {$f_0$};
	 			\draw[->] (beta0) -- (beta1);
	 			\draw[->] (beta1) -- (beta2);
	 			\draw[->] (alpha0) -- (beta0);
	 		\end{tikzpicture}.
	 			 	\end{equation}
	 	Here we are denoting $f_{2}=\psi_1$ and  $f_3=\psi_2$. 
	 	In this diagram, the horizontal arrows are nonzero, otherwise $\psi_0$ would not be linearly full. Now, if the vertical arrow is nonzero, then $f_0:S^2\to\CP^3$ is a linearly full holomorphic map.  Since $f_2=\psi_1$ has constant curvature, then, by \cite[Theorem 5.4]{Bolton},  $f_0$ is unitarily congruent to the Veronese map $V^3_0$, and, consequently, $\Psi$ is unitarily congruent to $\Psi_1=  \left( V_{0}^{3} \oplus V_{1}^{3},   V_{2}^{3},  V_{3}^{3}\right)$. Since, by Theorem \ref{veronesebolton}, 
	 	$\gamma^{3}_{1}= \frac{4}{(1+z\bar{z})^2}$ and $\gamma^{3}_{2}= \frac{3}{(1+z\bar{z})^2}$,
	 	the metric induced on $S^2$ by $\Psi_1$ is 
	 		 $\frac{4 \lambda_{0} +3 \lambda_{1}}{(1+z\bar{z})^2}dzd\bar z$, 
	 	and, in view of \eqref{kpsi}, the corresponding curvature is given by 
	 		$K(\Psi_1)= \frac{4}{4 \lambda_{0}+ 3 \lambda_{1}}.$
	 	
	 	Now, if the vertical arrow vanishes, then $f_0$ is a constant one-dimensional complex subspace, which we denote by $\C$, and  the diagram \eqref{diag:G10} becomes
	 	\begin{equation*}
	 	\label{diag:G12}
	 			\begin{tikzpicture}
	 			\node (beta0) at (0, 1.1) {$f_1$};
	 			\node (beta1) at (1.1, 1.1) {$f_2$};
	 			\node (beta2) at (2.2, 1.1) {$f_3$};
	 			\node (alpha0) at (0, 0) {$\C$};
	 			\draw[->] (beta0) -- (beta1);
	 			\draw[->] (beta1) -- (beta2);
	 		\end{tikzpicture}.
	 	\end{equation*}
	 	In this case, $f_1$ is holomorphic. Again, since $f_2=\psi_1$ has constant curvature, then, by \cite[Theorem 5.4]{Bolton},  $f_1$ is unitarily congruent to the Veronese map $V^2_0$, and, consequently, $\Psi$ is unitarily congruent  to $\Psi_2=  \left( V_{0}^{2} \oplus \C, V_{1}^{2}, V_{2}^{2} \right)$.
	 	Since
	 		$\gamma_0^{2}=\gamma_1^{2} =\frac{2}{(1+z\bar{z})^2}$,
	 	the induced metric is given by
	 $\frac{2(\lambda_{0}+\lambda_{1})}{(1+z\bar{z})^2}dzd\bar z$;
	 	and the corresponding curvature 
	 		$K(\Psi_2) = \frac{2}{\lambda_{0}+ \lambda_1}.$

	 	Suppose now that $\psi_0$ has isotropy order $r=2$. We have the  fundamental harmonic diagram
	 	\begin{equation}\label{diagramisotropy2}
	 		\begin{tikzpicture}
	 				\node (beta0) at (0, 0) {$\psi_0$};
	 			\node (beta1) at (1.1, 0) {$\psi_1$};
	 			\node (beta2) at (2.2, 0) {$\psi_2$};
	 			\draw[->] (beta0) -- (beta1);
	 			\draw[->] (beta1) -- (beta2);
	 			\draw[->, bend right=30] (beta2) to node[swap] {} (beta0);
	 		\end{tikzpicture}
	 	\end{equation}
	 	with $\rank \psi_0=2$ and $\rank \psi_1=\rank\psi_2=1$. Observe that here we cannot apply  \cite[Theorem 2.1]{PM}  to conclude that each $\psi_i$ has constant curvature, because $\psi_0$ is not holomorphic.

	 	In \eqref{diagramisotropy2}, the backward arrow is nonzero, because $\psi_0$ has isotropy order $r=2$. 
	 	Set $f_0= \underline{\ker} A'_{\psi_0}$ and $f_{1}= f_0^\perp \cap \psi_{0}$, so that $\psi_0=f_0\oplus f_1$. By \cite[Proposition 2.2]{burstall-wood}, $f_0= \underline{\ker} A'_{\psi_0}$ is a holomorphic subbundle of $\psi_{0}$. 
	 Moreover, due to the topology of $S^2$, we know that the cycle \eqref{diagramisotropy2} is nilpotent \cite{PW1}, i.e., $c=A'_{\psi_2}\circ A'_{\psi_1}\circ A'_{\psi_0}$ is nilpotent.	Hence, we have the following refinement of \eqref{diagramisotropy2}:
	 	\begin{equation}\label{diagramisotropy2ref1}
	 		\begin{tikzpicture}
	 			\node (beta0) at (0, 1.1) {$f_1$};
	 			\node (beta1) at (1.1, 1.1) {$f_2$};
	 			\node (beta2) at (2.2, 1.1) {$f_3$};
	 			\node (alpha0) at (0, 0) {$f_0$};
	 			\draw[->] (beta0) -- (beta1);
	 			\draw[->] (beta1) -- (beta2);
	 \draw[->] (alpha0) -- (beta0);
	 			\draw[->] (beta2) to[out=250, in=-15, looseness=1.25] (alpha0);
	 		\end{tikzpicture}.
	 	\end{equation}By  \cite[Proposition 1.6]{burstall-wood}\footnote{The correction of  \cite[Proposition 1.6]{burstall-wood} can be consulted in  https://people.bath.ac.uk/feb/papers/bw-corr/paper.pdf}, the smooth map $f_1:S^2\to \CP^{3}$ is harmonic.
	 	If all the arrows in this diagram are nonzero, then $f_1:S^2\to \CP^{3}$ would be a harmonic map of finite isotropy order, which we know to be impossible due to the topology of $S^2$ (see \cite[Proposition 1.8]{burstall-wood}). Hence, in \eqref{diagramisotropy2ref1}, the vertical arrow  from $f_0$ to $f_1$ must be zero, and we get 
	 	\begin{equation*}\label{refinement of the diagram}
	 		\begin{tikzpicture}
	 			\node (beta0) at (0, 1.1) {$f_1$};
	 			\node (beta1) at (1.1, 1.1) {$f_2$};
	 			\node (beta2) at (2.2, 1.1) {$f_3$};
	 			\node (alpha0) at (0, 0) {$f_0$};
	 			\draw[->] (beta0) -- (beta1);
	 			\draw[->] (beta1) -- (beta2);
	 			\draw[->] (beta2) to[out=250, in=-15, looseness=1.25] (alpha0);
	 		\end{tikzpicture}.
	 	\end{equation*}
	 	
	 	In view of \eqref{eq:metric1}, the metric induced on $S^2$ by $\Psi=(f_0\oplus f_1,f_2,f_3)$ from the invariant metric on $F_{2,1,1}$ with parameters $\lambda_0,\lambda_1,\lambda_2$ is equal to the metric  induced on $S^2$ by $\tilde\Psi=( f_1,f_2,f_3,f_0)$ from the metric  on $F_{1,1,1,1}$ with parameters $\lambda_0,\lambda_1,\lambda_2,\lambda_3$, regardeless of the value of $\lambda_3>0$. Hence $\tilde\Psi$ is an immersion of constant curvature. By   \cite[Theorem 2.1]{PM}, we conclude that 
	 	$f_2=\psi_1:S^2\to\CP^3$ has constant curvature, then, by \cite[Theorem 5.4]{Bolton},  $f_1$ is unitarily equivalent to the Veronese map $V_0^3$, hence $f_0$ is unitarily congruent to $V_3^3$,
	 	and the result follows.
	 \end{proof}	
	
		\begin{remark}
		A careful inspection of the proof of Theorem \ref{mainresult1} shows that any other primitive immersion $\Psi:S^2\to F_{2,1,1}$ of constant curvature with respect to at least one $\U(4)$-invariant metric on $F_{2,1,1}$ can be obtained from Veronese primitive immersions by operating with homogeneous projections,  adding constants, adding primitive maps,  and shifts.  For example, in diagram \eqref{diagramisotropy2ref1}, the case where the arrows $f_0\to f_1$ and $f_2\to f_3$ vanish leads to the primitive immersion into $F_{2,1,1}$ of constant curvature given by $\Psi_1\oplus \Psi_2$, where 
		$\Psi_1=(V_0^1,V_1^1,0)$ and 	$\Psi_2=\tau^2\Psi_1=(V_1^1,0,V_0^1)$.
	\end{remark}
	
	 \begin{remark}\label{liftinF_2,1,1,1}
	 		More generally, suppose that $\Psi=\left(\psi_{0},\psi_{1},\ldots,\psi_{n-1}\right):S^2 \to F_{2,1,\ldots,1}$ is  a primitive immersion of constant curvature with respect to at least one $\U(n)$-invariant metric on $F_{2,1,\ldots,1}$, where $\psi_{0}:S^2 \to G_{2}(\C^n)$ is a linearly full harmonic map of isotropy order $r\geq n-1$, and each $\psi_j$ is the $j$-th $\partial'$-Gauss bundle of $\psi_0$. By straightforward extension of the previous proof,  $\Psi$ is unitarily congruent to one of the following curves:
	 		\begin{enumerate}
	 			\item  $\Psi_1=  \left( V_{0}^{n-1} \oplus V_{1}^{n-1},   V_{2}^{n-1},\ldots ,  V_{n-1}^{n-1}\right)$;
	 			\item $\Psi_2=  \left( V_{0}^{n-2} \oplus \C, V_{1}^{n-2},\ldots, V_{n-2}^{n-2} \right)$;
	 			\item  $\Psi_3=  \left( V_{0}^{n-1} \oplus V_{n-1}^{n-1},   V_{2}^{n-1},\ldots   V_{n-2}^{n-1}\right)$.
	 	\end{enumerate}
	 \end{remark}
	 
	 In the following remarks we compare our result with  \cite{Li99}.
	 
	 \begin{remark}\label{Lihol}
	 		In \cite[Theorem A]{Li99}, the authors provided a complete classification of holomorphic immersions of constant curvature $K$ from $S^2$ into $G_2(\C^4)$. 
	 		Any such immersion $\psi_0:S^2\to G_2(\C^4)$ is unitarily congruent to one of the following minimal immersions (locally spanned by $f_0$ and $g_0$, in terms of the canonical complex coordinate $z$ on $\C=S^2\setminus\{0\}$):
	   
	 (1)
	 			$f_0= (1,0,z,0)$ and $g_0=(0,1,0,0)$, with $K=4$.
	 			This holomorphic immersion is not linearly full. 
	 		
	 		(2) $f_0=(1,0,z^2 \cos 2t ,\sqrt 2 z \sin t )$ and $g_0=(0,1,\sqrt 2 z\cos t,0 )$
	 			for some $t\in [0,\pi)$, with $K=2$. Clerarly, $\psi_0$ is not lineally full if $t=0$, so we take  $t\in (0,\pi)$.
	 		A	straightforward computation shows that
	 			$\|f_0\wedge g_0\wedge \partial_z f_0\wedge \partial_z g_0\|^2=2\sin^2 t\cos^2 t.$
	 			Hence the first $\partial'$-Gauss bundle $\psi_1$ of $\psi_0$ has rank 1 if $t=\frac{\pi}{2}$ and it has rank 2 otherwise. For $t=\frac{\pi}{2}$, we have $f_0=V_0^2$ and $g_0$ is constant, hence  $\psi_0$ (with  $t=\frac{\pi}{2}$) generates the primitive lift $\Psi_2$ of Theorem \ref{mainresult1}.  
	 			
	 			(3) 
	 			$f_0= (1,0,\sqrt{3}\,z^2,0)$ and  $g_0=(0,1,\sqrt{8/3}\,z,\sqrt{1/3}\,z),$ with $K=\frac43$.
	 			A straightforward computation shows that
	 			$\|f_0\wedge g_0\wedge \partial_zf_0\wedge \partial_zg_0\|^2=4|z|^2.$
	 			Hence the first Gauss bundle of $\psi$ has rank $2$.
	 			
	 			(4)
	 			$f_0= (1,0,2z^3,\sqrt{3}z^2)$ and $g_0=(0,1,\sqrt{3}z^2,2z)$, with $K=1$.
	 			Consider the Plücker embedding $\iota: G_2(\C^4)\hookrightarrow  \CP^5$. A straightforward computation shows that, up to unitary congruence, 
	 			$\iota\circ \psi_0=\iota\circ (V_0^3\oplus V_1^3).$ 
	 			Hence $\psi_0$ generates the primitive immersion $\Psi_1$ in Theorem \ref{mainresult1}.
	 
	 \end{remark}

	 \begin{remark}\label{Linon}
	 	In \cite[Theorem B]{Li99}, the authors provided a complete classification for constant curved minimal immersions  $\psi_0:S^2\to G_2(\C^4)$ which are neither holomorphic nor antiholomorphic. Any such immersion is unitarily congruent to one of the following minimal immersions (locally spanned by $f_0$ and $g_0$, in terms of the canonical complex coordinate $z$ on $\C=S^2\setminus\{0\}$):
	 		
	 		(1)  $f_0(z)=(1,z,0,0)$ and $g_0(z)=(0,0,1,\bar z)$, with $K=2$. This minimal immersion has second  $\partial'$-Gauss bundle $\psi_2=\{0\}$.
	 		
	 		(2)  $f_0(z)=(1,\sqrt 2 \,z,z^2,0)$ and $g_0(z)=(\bar z^2,-\sqrt 2 \bar z,1,0)$, with $K=1$. This minimal immersion is not full. 
	 			
	 			(3) $f_0(z)=(1,\sqrt 3 \,z,\sqrt 3\,z^2,z^3)$ and $g_0(z)=(\bar z^3,-\sqrt 3 \bar z ^2,\sqrt 3\bar z,-1)$, with $K=\frac23$. This is precisely $\psi_0=V_0^3\oplus V_3^3$, and its harmonic sequence yields $\Psi_3$ in Theorem \ref{mainresult1}.
	 	
	 	(4)  $f_0(z)=(1,\sqrt 3 \,z,\sqrt 3\,z^2,z^3)$ and $g_0(z)=(\sqrt 3 \bar z^2,\bar z(z\bar z-2),1-2z\bar z,\sqrt 3 z)$, with $K=\frac25$. We have
	 			$\|f_0\wedge g_0\wedge \partial_zf_0\wedge \partial_zg_0\|^2=9(1+z\bar z)^8.$
	 			Hence the first $\partial'$-Gauss bundle has rank $2$. 
	 		
	 \end{remark}
	 
	 From Remark \ref{Lihol} and Remark \ref{Linon} we conclude that our results are consistent  with the classification results in \cite{Li99}.
	 \section{Primitive immersions into $F_{2,2,1}$}\label{F221}
	 
	 \begin{theorem}\label{mainresult2}
	 	Let $\Psi= \left(\psi_0,\psi_1, \psi_2\right):S^2 \to F_{2,2,1}$ be a primitive immersion of constant curvature with respect to at least one $\U(5)$-invariant metric on $F_{2,2,1}.$ Here  $\psi_0:S^2 \to G_{2}(\C^5)$ is a  linearly full holomorphic map with 
	 	first and second $\partial '$-Gauss bundles $\psi_1$ and $\psi_2$. Then, $\Psi$ is unitarily congruent to the curve 
	 	\begin{equation}\label{primitivemainresult}
	 		\Psi= \left(V_{0}^{2} \oplus V_{0}^{1}, V_{1}^{2} \oplus V_{1}^{1}, V_{2}^{2}\right)\qquad \mbox{ (taking  $\C^5=\C^3\oplus \C^2$)}.
	 	\end{equation}
	 	For a $\U(5)$-invariant metric on $F_{2,2,1}$ with parameters $\lambda_0,\lambda_1,\lambda_2$, using the notations of \eqref{eq:metric1}, the  metric  induced by $\Psi$ has  curvature
	 	$K(\Psi)=\frac{4}{3 \lambda_{0}+2 \lambda_{1}}.$
	 \end{theorem}
	 \begin{proof}
	 	Suppose that the immersion $\Psi=(\psi_0,\psi_1,\psi_2)$ has constant curvature with respect to at least one  $\U(5)$-invariant metric on $F_{2,2,1}$. Since $\psi_{0}:S^2 \to G_{2}(\C^5)$ is holomorphic, then by  \cite[Theorem 2.1]{PM},  $\Psi$ has constant curvature w.r.t all $\U(5)$-invariant metrics on $F_{2,2,1}$ and each $\psi_{i}$ is an immersion of constant curvature. We have the  following fundamental harmonic diagram
	 	\begin{equation*}\label{second harmonic sequence}
	 	\begin{tikzpicture}
	 		\node (beta0) at (0, 0) {$\psi_0$};
	 		\node (beta1) at (1.1, 0) {$\psi_1$};
	 		\node (beta2) at (2.2, 0) {$\psi_2$};
	 		\draw[->] (beta0) -- (beta1);
	 		\draw[->] (beta1) -- (beta2);
	 	\end{tikzpicture}.
	 \end{equation*}
	 	with $\rank \psi_0=\rank \psi_1=2$ and
	 	$\rank\psi_2=1$.
	 	Now, define the line subbundles
	 	$$f_2=\underline{\ker}A'_{\psi_1},\quad f_3=f_2^\perp\cap \psi_1,\quad f_1=\underline{\Img} A''_{\psi_1}{|_{f_3}},\quad f_0=  f_1^\perp\cap \psi_0,\quad f_4=\underline{\Img}A'_{\psi_1}.$$
	 	Then we have the following refinement of the fundamental diagram
	 	\begin{equation}\label{diag:G26}
	 			\begin{tikzpicture}
	 			\node (beta0) at (0, 1.1) {$f_1$};
	 			\node (beta1) at (1.1, 1.1) {$f_3$};
	 			\node (beta2) at (2.2, 1.1) {$f_4$};
	 			\node (alpha0) at (0, 0) {$f_0$};
	 			\node (alpha1) at (1.1, 0) {$f_2$};
	 			\draw[->] (beta0) -- (beta1);
	 			\draw[->] (beta1) -- (beta2);
	 			\draw[->] (beta0) -- (alpha1);
	 			\draw[->] (alpha0) -- (alpha1);
	 			\draw[->] (alpha0) -- (beta0);
	 			\draw[->] (alpha1) -- (beta1);
	 		\end{tikzpicture}
	 	\end{equation}
	 	where $\psi_0=f_0\oplus f_1$,  $\psi_1=f_2\oplus f_3$ and $\psi_2=f_4$. Due to our assumptions on the harmonic sequence $\psi_0,\psi_1,\psi_2$, all the horizontal arrows are nonzero. Observe that $f_3$ and $f_4$ are linearly full harmonic maps in $\CP^{p}$, for some $p\in\{2,3,4\}$. On the other hand, since $\psi_2=f_4$ has constant curvature, 
	 	$f_3,f_4:S^2\to \CP^{p}$ belong to a Veronese sequence: $f_3=V_{p-1}^{p}$ and $f_4=V_p^{p}$, up to unitary congruence. We analyze these three cases separately.

	 	\textbf{Case $p=4$.} We have 
	 	\begin{equation}\label{diag:G27}
	 			\begin{tikzpicture}
	 			\node (beta0) at (0, 1.1) {$f_1$};
	 			\node (beta1) at (1.1, 1.1) {$V_3^4$};
	 			\node (beta2) at (2.2, 1.1) {$V_4^4$};
	 			\node (alpha0) at (0, 0) {$f_0$};
	 			\node (alpha1) at (1.1, 0) {$f_2$};
	 			\draw[->] (beta0) -- (beta1);
	 			\draw[->] (beta1) -- (beta2);
	 			\draw[->] (beta0) -- (alpha1);
	 			\draw[->] (alpha0) -- (alpha1);
	 			\draw[->] (alpha0) -- (beta0);
	 			\draw[->] (alpha1) -- (beta1);
	 		\end{tikzpicture}.
	 	\end{equation}
	 	Now, take the composition of $\psi_0$ with the Pl\"{u}cker embedding  
	 	$G_{2}(\C^5)\hookrightarrow \CP^9$  to obtain  a holomorphic map 
	 	$\sigma_0:S^2\to\CP^{9}$. Since $\psi_0$ has constant curvature, then $\sigma_0$ also has constant curvature and it coincides, up to unitary congruence, with the Veronese map in $\CP^k$, for some $k\in\{1,\ldots, 9\}$.   Let $\hat\sigma_0:\C \to \C^{10}$ be a holomorphic local section of $\sigma_0$. Without loss of generality, we assume that   $\hat \sigma_0$ is nonvanishing and that its components are holomorphic polynomials.   
	 	Then 
	 	\begin{equation}\label{constant curvature}
	 		\|\hat\sigma_0\|^2=  \left(1+z\bar{z}\right)^k \quad \mbox{for some  $k\in\{1,\ldots, 9\}$}.
	 	\end{equation}
	 	
	 	Consider the local section $v_0=\left(1, 2z, \sqrt{6}z^2, 2z^3,z^4\right)$  of $V_0^{4}.$ In view of \eqref{diag:G27}, we have
	 	\begin{align*}
	 		f_0\oplus f_1\oplus f_2=V_0^{4}\oplus V_1^{4}\oplus V_2^{4} =\underline{\Span}\{v_0,\partial_z v_0, \partial_z^2 v_0\};
	 	\end{align*}
	 	hence
	 	\begin{equation}\label{f_0f_1}
	 		\begin{aligned}
	 			\hat\sigma_0 &= P_0 \left(v_{0}\wedge\partial_z v_{0}\right) +P_1\left(v_{0}\wedge \partial_z^{2}v_0\right) + P_2\left(\partial_z v_{0}\wedge \partial_z^{2}v_0\right)
	 		\end{aligned}
	 	\end{equation}
	 	for some polynomials $P_0,P_1,P_2$ in $z$. Consider the local sections $v_1$ and $v_2$ of $V_1^{4}$ and $V_2^{4}$, respectively, defined as in \eqref{veronesesection}. From \eqref{sectionfarmoulaveronese}, we have the following  
	 	\begin{equation}\label{partialV0}
	 		\partial_zv_0=v_1+\frac{4\bar z}{1+z\bar z}v_0,\quad \partial_zv_1=v_2+\frac{2\bar z}{1+z\bar z}v_1.
	 	\end{equation}
	 	By using \eqref{partialV0}, 	after straightforward computation we deduce that
	 	\begin{equation*}\label{V_0}
	 		\partial_z^2 v_{0}=v_2+\frac{6\bar{z}}{(1+z\bar{z})}v_1+ \frac{12 \bar{z}^2}{(1+z\bar{z})^2}v_{0}. 
	 	\end{equation*}
	 	Hence,
	 	\begin{align*}
	 		& v_{0}\wedge\partial_z v_{0} =  v_{0}\wedge v_{1},\quad v_{0}\wedge \partial_z^2 v_0 = \frac{6\bar{z}}{1+z\bar{z}} v_0 \wedge v_1+ v_0 \wedge v_2,
	 		\\ &	\partial_z v_0 \wedge \partial_z^{2} v_{0} = \frac{12\bar{z}^2}{(1+z\bar{z})^2}v_0 \wedge v_1 +\frac{4\bar{z}}{(1+z\bar{z})}v_0 \wedge v_2 + v_1 \wedge v_2.
	 	\end{align*}
	 	Replacing these formulas in \eqref{f_0f_1},  we obtain
	 	\begin{equation*}\label{wedge f_0 f_1}
	 		\hat\sigma_0 = \left(P_0 + \frac{6 \bar{z}P_1}{1+z\bar{z}}+\frac{12 \bar{z}^2P_2}{(1+z\bar{z})^2} \right)v_{0} \wedge v_{1}+ P_{2}\left(v_1 \wedge v_2\right) +\left(P_1+\frac{4\bar{z}P_2}{1+z\bar{z}}\right)v_0 \wedge v_2.
	 	\end{equation*}
	 	Taking the norm square (recall that $v_0,v_1,v_2$ are mutually orthogonal) 
	 	\begin{equation}\begin{aligned}\label{norm}
	 			\left\|\	\hat\sigma_0 \right\|^2 = \left| P_0 + \frac{6 \bar{z}P_1}{1+z\bar{z}}+\frac{12 \bar{z}^2P_2}{(1+z\bar{z})^2}\right|^2&\left\|v_0\right\|^2 \left\|v_1\right\|^2 +\left\|P_2\right\|^2 \left\|v_1\right\|^2\left\|v_2\right\|^2\\ &+ \left|P_1+\frac{4\bar{z}P_2}{1+z\bar{z}}\right|^2\left\|v_2\right\|^2\left\|v_0\right\|^2.\end{aligned}
	 	\end{equation}
	 	From Theorem \ref{veronesebolton} we know that
	 	$$\frac{\|v_1\|^2}{\|v_0\|^2}=\gamma^{4}_0=\frac{4}{(1+z\bar z)^2},\qquad\frac{\|v_2\|^2}{\|v_1\|^2}=\gamma^{4}_1=\frac{6}{(1+z\bar z)^2}.$$
	 	Together with \eqref{norm}, and since $\|v_0\|^2= (1+z\bar z)^4 $, this gives
	 	\begin{equation}\label{metric}
	 		\begin{aligned}
	 			\left\|	\hat\sigma_0 \right\|^2  &=4 \left| P_0\left(1+z\bar{z}\right)^3 + {6 \bar{z}}{(1+z\bar{z})^2}P_1+{12 \bar{z}^2}{(1+z\bar{z})}P_2\right|^2\\ &+ 96 \left|\left(1+z\bar{z}\right)P_2\right|^2+24 \left|{(1+z\bar{z})^2}P_1 + 4\bar{z}\left(1+z\bar{z}\right)P_2\right|^2.
	 		\end{aligned}
	 	\end{equation}
	 	Combining (\ref{constant curvature}) and (\ref{metric}),
	 	\begin{equation}\label{metric in terms of veronese}
	 			\left(1+z\bar{z}\right)^{k-2}=4 \left| P_0\left(1+z\bar{z}\right)^2 + {6 \bar{z}}{(1+z\bar{z})}P_1+{12 \bar{z}^2}P_2\right|^2 +96 \left|P_2\right|^2 + 24 \left|{(1+z\bar{z})}P_1 + 4\bar{z}P_2\right|^2.
	 	\end{equation}
	 	
	 	The case $k=1$ is impossible, because the right-hand side of \eqref{metric in terms of veronese} is polynomial in $z,\bar z$. 
	 		For $k=2$, equation \eqref{metric in terms of veronese} yields
	 	\begin{equation*}
	 			1=4 \left| P_0\left(1+z\bar{z}\right)^2 + {6 \bar{z}}{(1+z\bar{z})}P_1+{12 \bar{z}^2}P_2\right|^2 +96 \left|P_2\right|^2 + 24 \left|{(1+z\bar{z})}P_1 + 4\bar{z}P_2\right|^2.
	 	\end{equation*}
	 	From this we see that the three adding terms on the right-hand side must be bounded. But this occurs if, and only if, all the polynomials $P_0(z)$, $P_1(z)$ and $P_2(z)$ are  zero, giving a contradiction.
	 	
	 Suppose that $k>2$. Recall that $\C[z,\bar{z}]$ is a unique factorization domain and $(1+z\bar{z})$ is irreducible in $\C[z,\bar{z}]$.
	 	Observe that if we expand the right-hand side of \eqref{metric in terms of veronese}, all the terms except $\left|P_2\right|^2\left(96+384 z\bar{z}+576z^2\bar{z}^2\right)$ are divisible by $\left(1+z\bar{z}\right).$
	 	Since, for $k>2$, the factor $\left(1+z\bar{z}\right)$ divides the left-hand side of \eqref{metric in terms of veronese},  it must divide  
	 	$\left|P_2\right|^2\left(96+384 z\bar{z}+576z^2\bar{z}^2\right)$  as well. 	But $P_2$ is a holomorphic polynomial, hence, if $P_2\neq 0$, $\left|P_2\right|^2$ is not divisible by  $(1+z\bar{z})$; moreover, the polynomial  $96+384 z\bar{z}+576z^2\bar{z}^2$ is not divisible by $(1+z\bar{z})$. Hence, we must have $P_2=0$.
	 	In this case, equation \eqref{metric in terms of veronese} simplifies to
	 	\begin{equation}\label{norm2}
	 		\begin{aligned}
	 			\left(1+z\bar{z}\right)^{k-4}=4\left|P_0\left(1+z\bar{z}\right)+6\bar{z}P_1\right|^2+24\left|P_1\right|^2.
	 		\end{aligned}
	 	\end{equation}
	 	From the equation \eqref{norm2}, observe that for $k=3$ the left-hand side is no longer  polynomial, whereas the right-hand is  polynomial. Therefore, $k=3$ is not possible.
	 	
	 	For $k=4$, equation \eqref{norm2} yields
	 $$1= 4\left|P_0\left(1+z\bar{z}\right)+6\bar{z}P_1\right|^2+24\left|P_1\right|^2.$$
	 From this we see that both terms $\left|P_1\right|^2$ and $\left|P_0\left(1+z\bar{z}\right)+6\bar{z}P_1\right|^2$ on the right-hand side must be bounded. But this occurs if, and only if, the polynomials $P_0(z)$ and $P_1(z)$  are both zero, giving a contradiction.
	 	
	 	 For $k>4$ we must have $P_1=0$ by applying the similar arguments as those used to show that $P_2=0$. Consequently,  $\hat \sigma_0= P_0\left(v_0 \wedge v_1\right)$. This implies that $$\psi_0=f_0\oplus f_1=V^{4}_0\oplus V^{4}_1,$$
	 	which leads to the following refinement of the diagram \eqref{diag:G27}
	 	\begin{equation}\label{lotsV}
	 			\begin{tikzpicture}
	 				\node (beta0) at (0, 1.1) {$V_1^4$};
	 				\node (beta1) at (1.1, 1.1) {$V_3^4$};
	 				\node (beta2) at (2.2, 1.1) {$V_4^4$};
	 				\node (alpha0) at (0, 0) {$V_0^4$};
	 				\node (alpha1) at (1.1, 0) {$V_2^4$};
	 				\draw[->] (beta1) -- (beta2);
	 				\draw[->] (beta0) -- (alpha1);
	 				\draw[->] (alpha0) -- (beta0);
	 				\draw[->] (alpha1) -- (beta1);
	 			\end{tikzpicture}.
	 	\end{equation}
	 	Here the first $\partial^{\prime}$-Gauss bundle of $\psi_{0}$ is $V^{4}_2$,
	 	leading to a contradiction, since we are assuming that the first $\partial'$-Gauss bundle of $\psi_0$ has rank 2.

	 	\textbf{Case $p=3$}. Let us consider the case  $f_3=V_2^{3}$. The diagram \eqref{diag:G26} will then take the following form
	 	\begin{equation}\label{diag:G28}
	 			\begin{tikzpicture}
	 				\node (beta0) at (0, 1.1) {$f_1$};
	 				\node (beta1) at (1.1, 1.1) {$V_2^3$};
	 				\node (beta2) at (2.2, 1.1) {$V_3^3$};
	 				\node (alpha0) at (0, 0) {$f_0$};
	 				\node (alpha1) at (1.1, 0) {$f_2$};
	 				\draw[->] (beta0) -- (beta1);
	 				\draw[->] (beta1) -- (beta2);
	 				\draw[->] (beta0) -- (alpha1);
	 				\draw[->] (alpha0) -- (alpha1);
	 				\draw[->] (alpha0) -- (beta0);
	 				\draw[->] (alpha1) -- (beta1);
	 			\end{tikzpicture}.
	 	\end{equation}
	 	Again, by taking the composition of $\psi_0=f_0\oplus f_1$ with the Pl\"{u}cker embedding  of
	 	$G_{2}(\C^5)\hookrightarrow \CP^9$  we obtain  a holomorphic map 
	 	$\sigma_0:S^2\to\CP^{9}$. 
	 	Since $\psi_0$ has constant curvature, then $\sigma_0$ also has constant curvature and it coincides with the Veronese map into $\C P^k$, for some $k\in\{1,\ldots,9 \}$. 
	 	Let $\hat\sigma_0:\C \to \C^{10}$ be a holomorphic local section of $ \sigma_0$, which we assume to be polynomial and  nonvanishing. Then \eqref{constant curvature} holds, that is, 
	 	$	\|\hat\sigma_0\|^2=  \left(1+z\bar{z}\right)^k $ for some  $k\in\{1,\ldots, 9\}$.

	 	Consider the orthoghonal decomposition 
	 	$\C^5=V_0^{3}\oplus V_1^{3}\oplus V_2^{3}\oplus V_3^{3}\oplus\Span\{e\},$
	 	where $e=(0,0,0,0,1)$. Consider the section $v_0=\left(1,\sqrt{3}z,\sqrt{3}z^2,z^3,0\right)$ of $V_0^{3}$, and let $v_1$, $v_2$ and $v_3$ be local sections of $V_1^{3}$, $V_2^{3}$, and $V_3^{3}$, respectively, defined by \eqref{veronesesection}.
	 	
	 	Observe that $\hat\sigma_0$ is a section of 
	 	$ \underline{\Span}\{v_{0}\wedge e, v_{0}\wedge \partial_z v_{0}, \partial_z v_{0} \wedge e\}.$  Hence, 
	 	\begin{equation}\label{s0 wedge s1}
	 		\hat\sigma_0= P_{0}( v_{0}\wedge\partial_z v_{0}) + P_{1} (v_{0}\wedge e) +P_{2} (\partial_z v_{0} \wedge e)
	 	\end{equation} where $P_{0}$, $P_{1}$ and $P_{2}$ are holomorphic polynomials. 
	 	Since, in view of \eqref{sectionfarmoulaveronese}, 
	 	$\partial_z v_0=v_1+\frac{3\bar z}{1+z\bar z}v_0,$
	 	from \eqref{s0 wedge s1} we get
	 	\begin{equation}\label{wedge product}
	 		\hat\sigma_0=P_{0}\left( v_{0}\wedge v_{1}\right)+ \left(P_{1}+\frac{3\bar z}{1+z\bar{z}}P_2\right)v_{0} \wedge e + P_{2}\left(v_{1} \wedge e\right). 
	 	\end{equation}
	 	Taking the norm square,
	 	\begin{equation}\label{norm1}
	 		\left\|\hat\sigma_0\right\|^2= \left|P_{0}\right|^2\left\|v_{0}\right\|^2\left\|v_{1}\right\|^2+ \left|P_{1}+\frac{3\bar z}{1+z\bar{z}}P_2\right|^2\left\|v_{0}\right\|^2+\left|P_{2}\right|^2\left\|v_{1}\right\|^2.
	 	\end{equation} 
	 	For $v_0=\left(1,\sqrt{3}z,\sqrt{3}z^2,z^3,0\right)$,  we have
	 	$\left\|v_0\right\|^2=\left(1+z\bar{z}\right)^3;$
	 	and Theorem \ref{veronesebolton} gives
	 	$$\frac{\left\|v_1\right\|^2}{\left\|v_0\right\|^2}=\gamma_0^3=\frac{3}{(1+z\bar z)^2}.$$
	 	Hence, from \eqref{norm1} we deduce that
	 	$$ \left(1+z\bar{z}\right)^k=3\left|P_{0}\right|^2\left(1+z\bar{z}\right)^4+\left|P_1\left(1+z\bar{z}\right)+3\bar{z}P_2\right|^2\left(1+z\bar{z}\right)+3\left|P_2\right|^2\left(1+z\bar{z}\right)$$
	 	for some $k\in\{1,\ldots,9\}$, hence
	 	\begin{equation}\label{mainequation}
	 		\left(1+z\bar{z}\right)^{k-1}=3\left|P_{0}\right|^2\left(1+z\bar{z}\right)^3+\left|P_1\left(1+z\bar{z}\right)+3\bar{z}P_2\right|^2 +3\left|P_2\right|^2.
	 	\end{equation}

	 	For $k=1$, this equation simplifies to 
	 	$$1=3\left|P_{0}\right|^2\left(1+z\bar{z}\right)^3+\left|P_1\left(1+z\bar{z}\right)+3\bar{z}P_2\right|^2 +3\left|P_2\right|^2.$$
	 	Again, the three terms on the right-hand side must be bounded, and this forces a contradiction. So we have $k\geq 2$.
	 	
	 	For $k\geq 2$, we see from \eqref{mainequation} that $\left(1+z\bar{z}\right)$ must divide $\left|P_2\right|^2$. But $P_2$ is a polynomial in $z$, hence $\left|P_2\right|^2$ is not divisible by $\left(1+z\bar{z}\right)$, unless $P_2=0$. Substituting this in \eqref{mainequation}, we obtain
	 	\begin{equation}\label{mainequation2}
	 		\left(1+z\bar{z}\right)^{k-3}=3\left|P_{0}\right|^2\left(1+z\bar{z}\right)+\left|P_1\right|^2
	 	\end{equation}
	 	From \eqref{mainequation2}, we see that $k\geq 3$, because the right-hand side is polynomial. If $k>3$, then $P_1=0$, with the same argument as above. Thus, 
	 	$\hat \sigma_0=P_0\left(v_0\wedge v_1\right)
	 	$, hence $\psi_0=f_0\oplus f_1=V_0^{3}\oplus V_1^{3}$. This implies that the first $\partial'$-Gauss bundle of $\psi_0$ has rank 1, which is a contradiction with the hypothesis of our theorem.  For $k=3$, equation \eqref{mainequation2} yields
	 	$$ 1=3\left|P_{0}\right|^2\left(1+z\bar{z}\right)+\left|P_1\right|^2.$$ 
	 	Since both adding terms on the right-hand side must be bounded, we see that $P_0=0$ and $P_1$ is a constant polynomial. Replacing $P_0=P_2=0$ in \eqref{s0 wedge s1}, we get $\hat\sigma_0=P_1(v_0\wedge e)$, hence $\psi_0=V_0^3\oplus\Span\{e\}$. In this case, we would have $\psi_1=V_1^3$,
	 	which has rank 1, leading to a contradiction with our assumptions.

	 		\textbf{Case $p=2$}. Let us consider the case   $f_3=V_1^{2}$. The diagram \eqref{diag:G26} will then take the following form
	 		\begin{equation}
	\label{diag:G32}
	 			\begin{tikzpicture}
	 				\node (beta0) at (0, 1.1) {$f_1$};
	 				\node (beta1) at (1.1, 1.1) {$V_1^2$};
	 				\node (beta2) at (2.2, 1.1) {$V_2^2$};
	 				\node (alpha0) at (0, 0) {$f_0$};
	 				\node (alpha1) at (1.1, 0) {$f_2$};
	 				\draw[->] (beta0) -- (beta1);
	 				\draw[->] (beta1) -- (beta2);
	 				\draw[->] (beta0) -- (alpha1);
	 				\draw[->] (alpha0) -- (alpha1);
	 				\draw[->] (alpha0) -- (beta0);
	 				\draw[->] (alpha1) -- (beta1);
	 			\end{tikzpicture}.
	 		\end{equation}
	 		We take the orthogonal decomposition $\C^5=\C^3\oplus \C^2$, with 
	 		$\C^3=V_0^{2}\oplus V_1^{2}\oplus V_2^{2}.$
	 		Since there is no arrow from $f_0$ to $V_1^{2}$, we see that $f_0$ has zero component along $V_0^2$, hence $f_0$ is a holomorphic line bundle in $\C^2$.  Let $g_0:\C\to \C^2$ be a local holomorphic section of  $f_0$,  and  consider the section $v_0=\left(1,\sqrt{2}z,z^2,0,0\right)$ of $V_0^{2}$.  Observe that  $g_0$ is not constant, otherwise the first $\partial'$-Gauss bundle of $\psi_0$ would have rank $1$;  then 
	 		$\C^2=\underline{\Span}\{g_0,\partial_z g_0\}.$
	 		
	 		Use the Pl\"{u}cker embedding  
	 		$G_{2}(\C^5)\hookrightarrow \CP^9$  to obtain  a holomorphic map 
	 		$\sigma_0:S^2\to\CP^{9}$. Let $\hat\sigma_0:\C \to \C^{10}$ be a polynomial nonvanishing  local section of $ \sigma_0$.  Denote by $e$ a constant unit vector spanning the image of $\C^2$ under the Pl\"{u}cker embedding. 
	 			We have
	 		\begin{equation}\label{p2casesigma}
	 			\hat\sigma_0=P_0\left(v_0 \wedge g_0\right)+P_1e
	 		\end{equation}
	 		for some polynomials $P_0$ and $P_1$. 
	 		Since $v_0=\left(1,\sqrt{2}z,z^2,0,0\right)$, then $\left\|v_0\right\|^2=\left(1+z\bar{z}\right)^2$. Taking the norm square in the above equation, we obtain
	 		\begin{equation}\label{norm CP2}
	 			\|\hat\sigma_0\|^2=\left|P_0\right|^2\left\|v_0\right\|^2\left\|g_0\right\|^2+\left|P_1\right|^2=\left(1+z\bar{z}\right)^2\left|P_0\right|^2\left\|g_0\right\|^2+\left|P_1\right|^2.
	 		\end{equation}
	 		On the other hand, since $\psi_0$ has constant curvature,  then $\|\hat\sigma_0\|^2=(1+z\bar z)^k$ for some $k\in\{1,\ldots,9\}$, hence	
	 		\begin{equation}
	 			\label{norm CP2}
	 			\left(1+z\bar{z}\right)^k= \left(1+z\bar{z}\right)^2\left|P_0\right|^2\left\|g_0\right\|^2+\left|P_1\right|^2. 
	 		\end{equation} 
	 		Hence $|P_1|^2$ must be divisible by $(1+z\bar z)$, which implies $P_1=0$.  Replacing this in \eqref{p2casesigma}	gives 	$\hat \sigma_0=P_0\left(v_0 \wedge g_0\right)$, therefore
	 		$\psi_0=f_0\oplus f_1 =f_0\oplus V^{2}_0,$ implying $f_1=V^{2}_0$, and this gives the following refinement of the diagram \eqref{diag:G32}
	 		\begin{equation}
	 			\label{diag:G32a}
	 				\begin{tikzpicture}
	 				\node (beta0) at (0, 1.1) {$V_0^2$};
	 				\node (beta1) at (1.1, 1.1) {$V_1^2$};
	 				\node (beta2) at (2.2, 1.1) {$V_2^2$};
	 				\node (alpha0) at (0, 0) {$f_0$};
	 				\node (alpha1) at (1.1, 0) {$f_2$};
	 				\draw[->] (beta0) -- (beta1);
	 				\draw[->] (beta1) -- (beta2);
	 				\draw[->] (alpha0) -- (alpha1);
	 			\end{tikzpicture}.
	 		\end{equation}
	 		Replacing $P_1=0$ in \eqref{norm CP2} we see that 
	 		$\|g_0\|^2=\frac{(1+z\bar z)^{k-2}}{|P_0|^2},$
	 		which implies, in view of \eqref{metricholocp}, that $f_0$ has constant curvature, hence $f_0=V_0^{1}$, up to unitary congruence. We conclude that
	 		$\Psi= \big(V_{0}^{2} \oplus V_{0}^{1}, V_{1}^{2} \oplus V_{1}^{1}, V_{2}^{2}\big).$
	 		The induced metric from the flag $F_{2,2,1}$ and the corresponding curvature are then given as follows
	 			$\frac{3 \lambda_{0}+2 \lambda_{1}}{\left(1+z\bar{z}\right)^2}dzd\bar z$, and $K(\Psi)=\frac{4}{3 \lambda_{0}+2 \lambda_{1}}$.
	 	\end{proof}

	 	\begin{remark}
	 			Let $\psi_{0}:S^2 \to G_{2}(\C^6)$ be a linearly full holomorphic map whose harmonic sequence $\psi_0,\psi_1,\psi_2,\psi_3$	gives rise to a primitive immersion $\Psi=\left(\psi_{0},\psi_{1},\psi_{2},\psi_{3}\right)$ into $F_{2,2,1,1}$ of constant curvature with respect to  at least  one $\U(6)$-invariant metric on $F_{2,2,1,1}$. The proof of Theorem \ref{mainresult2} can be straightforwardly adapted to conclude that
	 			$\Psi=(V_0^3\oplus V_0^1,V_1^3\oplus V_1^1,V_2^3,V_3^3),$ up to unitary congruence. However an open question remains corresponding the flag $F_{2,2,2}$.
	 	\end{remark}
	 	
	 	\begin{remark}
	 			In \cite[Theorem 4.2]{Jiao2004} and \cite{Jiao2011}, the authors classified all nonsingular holomorphic curves from $S^2$ into $G_2(\C^5)$ with constant curvature. They showed that the possible values for the constant curvature are $4,2,\frac43,1,\frac45.$ Next we compare their results with our Theorem \ref{mainresult2}. It is easy to check that the curves in \cite[Theorem 4.2]{Jiao2004} with constant curvature $K=4$ and $K=2$ are not linearly full. Moreover, the curves with constant curvature $K=\frac45$ and $K=1$ have the first $\partial^{\prime}$-Gauss bundle with  nonconstant curvature; consequently,  by  \cite[Theorem 2.1]{PM}, their primitive lifts are not of constant curvature. We study with more detail the remaining case. According to \cite[Theorem 4.2]{Jiao2004},
	 			any nonsingular holomorphic immersion $\psi_0:S^2 \to G_{2}(\C^5)$ of constant curvature $K=\frac43$ is unitarily congruent to one of the following minimal immersions (locally spanned by $f_0$ and $g_0$, in terms of the canonical complex coordinate $z$ on $\C=S^2\setminus\{0\}$):
	 			
	 			(1)  $f_0(z)=\left(1,0, a\sqrt{3-a^2}z,a^{-1}(a^2-1)^{3/2}z,az^2\right)$ and
	 				$g_0(z) =\left(0,1,0,a^{-1}z,0\right)$,     where $1\le a \le \sqrt{3}$.
	 				Straightforward computations show the following:
	 				$\|g_0 \wedge f_0\|^2= (1+z\bar{z})^3,$
	 				which confirms that	$\psi_0$ has constant curvature $K=\frac43$;  moreover 
	 				$$\|\partial_z g_0 \wedge \partial_z f_0\wedge g_0 \wedge f_0\|^2=3+4z\bar{z}-a^2(1+(-3+a^2)(z\bar{z})^2),$$
	 				which shows that the first $\partial'$-Gauss bundle $\psi_1$  has rank 2, and   
	 				\begin{align*}
	 					\gamma_{1}  = \partial_z \partial_{\bar z}\log\|\partial_z g_0 \wedge \partial_z f_0\wedge g_0 \wedge f_0\|^2= \frac{4(-3+a^2)(-1+a^2z\bar z(-3+a^2-z\bar{z}))}{(-3+a^2-4z\bar{z}+a^2(-3+a^2)(z\bar{z})^2)^2}.
	 				\end{align*}
	 				Hence $\psi_{1}: S^2 \to G_{2}(\C^5)$ is not of constant curvature, but for $a=1$. In this case,
	 				$f_0(z)=\left(1,0, \sqrt{2}z,0,z^2\right)$ and
	 				$g_0(z) =\left(0,1,0,z,0\right)$, hence 	$\psi_0= V_0^2 \oplus V_0^1,$ which agrees with \eqref{primitivemainresult}.

	 		(2)		 $f_0(z)=\left(1,0,az,0,\sqrt{a^4-3a^2+3}z^2\right)$, and $g_0(z)= \Big(0,1,0,\frac{1}{\sqrt{a^4-3a^2+3}}z,\sqrt{\frac{(2-a^2)^3}{a^4-3a^2+3}}\,z\Big),$ where $0\le a \le \sqrt{2}.$ 
	 				Similarly to the previous case, the $\partial^{\prime}$-Gauss bundle $\psi_{1}:S^2\to G_{2}(\C^5)$ has constant curvature if and only if $a=\sqrt{2}.$ This case  corresponds again to $\psi_0= V_0^2 \oplus V_0^1$. 
	 	\end{remark}

	 	Next we consider the finite isotropy order case.  The technique used in \cite{JiaoXu2018,Li99} will play an important role in this proof.
	 	\begin{theorem}\label{mainresult3}
	 		Let $\Psi=\left(\psi_{0}, \psi_{1},\psi_{2}\right):S^2 \to F_{2,2,1}$ be a primitive immersion, where $\psi_{0}: S^2 \to G_{2}(\C^5)$ is a linearly full harmonic map of isotropy order 2;   $\psi_1$ and $\psi_2$ are the 
	 		first and second $\partial '$-Gauss bundles of $\psi_0$. If there exists at least one $\U(5)$-invariant metric on $F_{2,2,1}$ corresponding to which $\Psi$ has constant curvature, then $\Psi$ is unitarily congruent to the curve
	 		\begin{equation}\label{isotropictheorem}
	 		\Psi=\left(V_{0}^{4}\oplus V_{3}^{4}, V_{1}^{4}\oplus V_{4}^{4} ,V_{2}^{4}\right);
	 		\end{equation}
	 		consequently, $\Psi$ has constant curvature with respect to all $\U(5)$-invariant metrics  on $F_{2,2,1}$.
	 		For a $\U(5)$-invariant metric on $F_{2,2,1}$ with parameters $\lambda_0,\lambda_1,\lambda_2$, using the notations of \eqref{eq:metric1}, the  metric  induced by $\Psi$ has  curvature
	 		$K(\Psi)=\frac{2}{2\lambda_{0}+3\lambda_{1}+3\lambda_{2}}.$
	 	\end{theorem}
	 	\begin{proof}
	 		We have the following fundamental harmonic diagram
	 		\begin{equation}\label{second harmonic sequence1}
	 			\begin{tikzpicture}
	 			\node (beta0) at (0, 0) {$\psi_0$};
	 			\node (beta1) at (1.1, 0) {$\psi_1$};
	 			\node (beta2) at (2.2, 0) {$\psi_2$};
	 			\draw[->] (beta0) -- (beta1);
	 			\draw[->] (beta1) -- (beta2);
	 			\draw[->, bend right=30] (beta2) to node[swap] {} (beta0);
	 		\end{tikzpicture}.
	 		\end{equation}
	 		Both $\psi_{0}$ and  $\psi_{1}$ have rank 2, while $\psi_{2}$ has rank 1. 
	 		The induced metric is given by
	 		\begin{equation}\label{flag metric}
	 			\Psi^*ds^2_{F_{2,2,1}}=\eta^2 dzd\bar{z},\quad \mbox{where $ \eta^2=(\lambda_{0}\gamma_{0}+ \lambda_{1} \gamma_{1}+\lambda_{2}\gamma_{2})$}.
	 		\end{equation}
	 		Define $\phi= \eta dz$, so that
	 $\Psi^*ds^2_{F_{2,2,1}}=\phi\bar \phi$.
	 
	 	Set	$\alpha_1 = \underline{\ker}A'_{\psi_1}$, and $\beta_{1} = \alpha_1^{\perp} \cap \psi_{1}$.
	 		Note that $\alpha_{1}$ is a holomorphic vector subbundle of $\psi_{1}$ and  $\beta_1$ is an antiholomorphic vector subbundle of $\psi_{1}$. Set $\beta_{2}=\psi_2$. Since $A''_{\psi_1}$ is an antiholomorphic isomorphism between $\psi_{1}$ and $\psi_{0}$, we can define
	 		$\beta_{0}=\underline{\Img} A''_{\psi_1}|_{\beta_1}$, and $\alpha_0=\beta_{0}^{\perp}\cap\psi_{0}.$ Again, $\alpha_{0}$ is a holomorphic vector subbundle of $\psi_{0}$ and  $\beta_0$ is an antiholomorphic vector subbundle of $\psi_{0}$.  Moreover, due to the topology of $S^2$, we know that the cycle \eqref{second harmonic sequence1} is nilpotent \cite{PW1}, i.e., $c=A'_{\psi_0}\circ A'_{\psi_2}\circ A'_{\psi_1}$ is nilpotent. Hence, we have the following refinement of the diagram (\ref{second harmonic sequence1}): 
	 		\begin{equation}\label{refinement of the diagram}
	 			\begin{tikzpicture}
	 				\node (beta0) at (0, 1.1) {$\beta_0$};
	 				\node (beta1) at (1.1, 1.1) {$\beta_1$};
	 				\node (beta2) at (2.2, 1.1) {$\beta_2$};
	 				\node (alpha0) at (0, 0) {$\alpha_0$};
	 				\node (alpha1) at (1.1, 0) {$\alpha_1$};
	 				\draw[->] (beta0) -- (beta1);
	 				\draw[->] (beta1) -- (beta2);
	 				\draw[->] (beta0) -- (alpha1);
	 				\draw[->] (alpha0) -- (alpha1);
	 				\draw[->] (alpha0) -- (beta0);
	 				\draw[->] (alpha1) -- (beta1);
	 				\draw[->] (beta2) to[out=250, in=-17, looseness=1.25] (alpha0);
	 			\end{tikzpicture}.
	 		\end{equation}
	 		We choose a unitary frame $\{e_{1},\ldots,e_{5}\}$, where $e_{1},e_2,e_3,e_4,e_{5}$  are local sections of  $\alpha_{1}$, $\beta_{1}$, $\beta_{2}$,  $\alpha_{0}$, $\beta_{0}$, respectively.
	 		In  view of \eqref{refinement of the diagram}, we deduce that the Maurer-Cartan $\mathfrak{u}(5)$-valued one-form $W=(\omega_i^j)$ corresponding to the above unitary frame is given as follows
	 	
	 		\begin{equation}\label{MC froms}
	 			\left[
	 			\begin{matrix}
	 				de_1 
	 				\\
	 				de_2
	 				\\
	 				de_3
	 				\\de_4
	 				\\de_5
	 			\end{matrix}
	 			\right] =
	 			\underbrace{	\left[
	 				\begin{matrix} 
	 					\sqrt{-1}\nu_1  & a_{1}^{2}\phi & 0 & b_{1}^{4} \overline{\phi} & b_{1}^{5}\overline{\phi}
	 					\\
	 					-\overline{a_{1}^{2}\phi} & \sqrt{-1}\nu_2  & a_{2}^{3}{\phi} & 0 & b_{2}^{5}\overline{\phi}
	 					\\
	 					0 & -\overline{a_{2}^{3}\phi} & \sqrt{-1}\nu_3  & a_{3}^{4}{\phi} & 0
	 					\\
	 					-\overline{b_{1}^{4}}\phi & 0 & -\overline{a_{3}^{4} {\phi}} & \sqrt{-1}\nu_4  & a_{4}^{5}{\phi}
	 					\\
	 					-\overline{b_{1}^{5}}\phi &-\overline{b_{2}^{5}}\phi & 0 & -\overline{a_{4}^{5}{\phi}} & \sqrt{-1}\nu_5		
	 				\end{matrix}
	 				\right]}_W
	 			\left[
	 			\begin{matrix}
	 				e_1 
	 				\\
	 				e_2
	 				\\
	 				e_3
	 				\\
	 				e_4
	 				\\e_5
	 			\end{matrix}
	 			\right]
	 		\end{equation}
	 		Here $\nu_{i}$ are real valued 1-forms; $a_i^j$ and $b_i^j$ are complex valued functions. Under our assumptions ($\psi_1$ and $\psi_2$ are the 
	 		first and second $\partial '$-Gauss bundles of $\psi_0$, with ranks $2$ and $1$, respectively), all horizontal arrows in the diagram \eqref{refinement of the diagram} are nonzero. The backward arrow from $\beta_2$ to $\alpha_0$ is also nonzero, because $\psi_0$ has finite isotropy order. Hence
	 		the complex functions $b_{1}^{4}$, $a_{2}^{3}$, $a_{3}^{4}$ and $b_{2}^{5}$ are not zero except at isolated points. Since $\Psi^*ds^2_{F_{2,2,1}}=\phi\bar \phi$,  we have the following
	 		normalization
	 		\begin{equation}\label{normalize components}
	 			\lambda_{1}\left|a_{2}^{3}\right|^2 + \lambda_{0}\left|b_{1}^{4}\right|^2 +\lambda_{0}\left|b_{1}^{5}\right|^2+\lambda_{0}\left|b_{2}^{5}\right|^2+\lambda_{2}\left|a_{3}^{4}\right|^2=1.
	 		\end{equation}
	 		
	 		By \cite[Proposition 1.6]{burstall-wood}, $\beta_{1}$ is harmonic. We claim that 
	 		$\beta_1$ is a linearly full harmonic map in $\CP^4$. In fact, suppose that $\beta_1$ takes values in $\CP^3$. In this case, we have $\C^4=\beta_1\oplus \beta_2\oplus \alpha_0\oplus C$ for some line bundle $C$, hence one of the following harmonic diagrams  hold:
	 		\begin{equation*}\label{second harmonic sequence}
	 			\begin{tikzpicture}[>=stealth, auto, node distance=1.1cm, ]
	 				\node (01) {$0$};
	 				\node (beta1) [right of=01] {$\beta_1$};
	 				\node (beta2) [right of=beta1] {$\beta_{2}$};
	 				\node (alpha0) [right of=beta2] {$\alpha_{0}$};
	 				\node (C) [right of=alpha0] {$C$};
	 				\node (02) [right of=C] {$0$};
	 				\draw[->] (beta1) to node {} (beta2);
	 				\draw[->] (beta2) to node {} (alpha0);
	 				\draw[->] (alpha0) to node {} (C);
	 				\draw[->] (C) to node {} (02);
	 				\draw[->] (01) to node {} (beta1);
	 			\end{tikzpicture},
	 			\qquad \quad \begin{tikzpicture}[>=stealth, auto, node distance=1.2cm, ]
	 				\node (01) {$0$};
	 				\node (C) [right of=01] {$C$};
	 				\node (beta1) [right of=C] {$\beta_{1}$};
	 				\node (beta2) [right of=beta1] {$\beta_{2}$};
	 				\node (alpha0) [right of=beta2] {$\alpha_0$};
	 				\node (02) [right of=alpha0] {$0$};
	 				\draw[->] (beta1) to node {} (beta2);
	 				\draw[->] (beta2) to node {} (alpha0);
	 				\draw[->] (alpha0) to node {} (02);
	 				\draw[->] (C) to node {} (beta1);
	 				\draw[->] (01) to node {} (C);
	 			\end{tikzpicture}.
	 		\end{equation*}
	 		Due to dimension constraints, either $\beta_1$ is holomorphic or $\alpha_0$ is antiholomorphic. Both of  cases lead to a contradiction, because all the horizontal arrows in diagram \eqref{refinement of the diagram} are nonzero, thus establishing our claim.  
	 		
	 		Now, since $\beta_1$ is a linearly full harmonic map in $\CP^4$, it belongs to the  harmonic sequence $g_0,g_1,g_2,g_3,g_4$ associated to some full holomorphic map $g_0:S^2\to \CP^4$. Clearly, we must have 
	 		$g_1=\beta_1$, $g_2=\beta_2$, and $g_3=\alpha_0$.  Choose a local
	 		unitary frame $E_1,E_2,E_3,E_4,E_5$ such that $E_j$ spans $g_{j-1}$, with $j\in\{1,2,3,4,,5 \}$, and 
	 		\begin{equation}\label{Relatng e and E}
	 			\left[
	 			\begin{matrix}
	 				e_1 
	 				\\
	 				e_2
	 				\\
	 				e_3
	 				\\e_4
	 				\\e_5
	 			\end{matrix}
	 			\right] =
	 			\underbrace{\left[
	 				\begin{matrix} 
	 					u & 0 & 0 & 0 & v
	 					\\
	 					0 & 1 & 0 & 0 & 0
	 					\\
	 					0 & 0 & 1 & 0 & 0
	 					\\
	 					0 & 0 & 0 &1 & 0
	 					\\
	 					-\bar{v} & 0 & 0 & 0 & \bar{u}
	 				\end{matrix}
	 				\right] }_{A}
	 			\left[
	 			\begin{matrix}
	 				E_1 
	 				\\
	 				E_2
	 				\\
	 				E_3
	 				\\
	 				E_4
	 				\\E_5
	 			\end{matrix}
	 			\right]
	 		\end{equation} 
	 		where $u$ and $v$ are complex functions satisfying $\left|u\right|^2$ +$\left|v\right|^2=1$. 
	 		
	 		We claim that $u=0$ and will prove it later. Assume now  $u=0$. Then we have $e_1=E_5$, $e_5=E_1$, and $e_i=E_j$ for $i,j=2,3,4.$ Hence, diagram \eqref{refinement of the diagram} simplifies to
	 		\begin{equation*}\label{refinement of the diagram1}
	 			\begin{tikzpicture}
	 				\node (g0) at (0, 1.1) {$g_0$};
	 				\node (g1) at (1.1, 1.1) {$g_1$};
	 				\node (g2) at (2.2, 1.1) {$g_2$};
	 				\node (g3) at (0, 0) {$g_3$};
	 				\node (g4) at (1.1, 0) {$g_4$};
	 				\draw[->] (g0) -- (g1);
	 				\draw[->] (g1) -- (g2);
	 				\draw[->] (g3) -- (g4);
	 				\draw[->] (g2) to[out=250, in=-17, looseness=1.25] (g3);
	 			\end{tikzpicture}
	 		\end{equation*} 
	 		with $\psi_0=g_0\oplus g_3$, $\psi_1=g_1\oplus g_4$, and $\psi_2=g_2$. 
	 		Consider the primitive map $\tilde \Psi=(g_0,g_1,g_2,g_3,g_4)$ into $F_{1,1,1,1,1}.$
	 		The metric  $\Psi^*ds^2_{F_{2,2,1}}=( \lambda_{0} \gamma_{0}+ \lambda_{1} \gamma_{1}+
	 		\lambda_{2} \gamma_{2})dzd\bar z,$
	 		coincides with the metric on $S^2$ induced by $\tilde\Psi$ from the  $\U(5)$-invariant metric on $F_{1,1,1,1,1}$ with parameters $\lambda_0,\lambda_1,\lambda_2,\lambda_0,\lambda_3$, regardless the value of $\lambda_3>0$.  Hence, $\tilde \Psi$ has constant curvature, which means,
	 		by \cite[Corollary 2.2]{PM}, that $\tilde\Psi$ is the $4$-Veronese primitive map, up to holomorphic isometry, inducing the metric
	 		$\tilde\Psi^*ds^2_{F_{1,1,1,1,1}}=(\lambda_0\gamma_0^4+\lambda_1\gamma_1^{4} +\lambda_2\gamma_2^{4}+\lambda_0\gamma_3^{4})dzd\bar z,$
	 		where the $\gamma_j^n$, with $j\in\{0,1,2,3 \}$, are defined by \eqref{metric for veronese}, which implies that
	 		$K(\Psi)=K(\tilde \Psi)=
	 		\frac{2}{4\lambda_{0}+3\lambda_{1}+3\lambda_{2}}.$

	 		Now the proof of claim that $u=0$. Since $g_0,g_1,g_2,g_3,g_4$ is a harmonic sequence in $\CP^4$, the Maurer-Cartan $\mathfrak{u}(5)$-valued one-form $\tilde{W}=(\tilde \omega_i^j)$  associated to the unitary moving frame $E_1,E_2,E_3,E_4,E_5$ assumes the standard form
	 	
	 		\begin{equation}\label{MC froms of E}
	 			\left[
	 			\begin{matrix}
	 				dE_1 
	 				\\
	 				dE_2
	 				\\
	 				dE_3
	 				\\dE_4
	 				\\dE_5
	 			\end{matrix}
	 			\right] =
	 			\underbrace{\left[
	 				\begin{matrix} 
	 					\sqrt{-1}\mu_1  & c_{1}^{2}\phi & 0 & 0 & 0
	 					\\
	 					-\overline{c_{1}^{2}\phi} & 	\sqrt{-1}\mu_2 & c_{2}^{3}\phi & 0 & 0
	 					\\
	 					0 & -\overline{c_{2}^{3}\phi} & \sqrt{-1}\mu_3   & c_{3}^{4}{\phi} & 0
	 					\\
	 					0 & 0 &	-\overline{c_{3}^{4}\phi}  & \sqrt{-1}\mu_4  & c_{4}^{5}{\phi}
	 					\\
	 					0 & 0 & 0 & -\overline{c_{4}^{5}{\phi}} & \sqrt{-1}\mu_5
	 				\end{matrix}
	 				\right]}_{\tilde W}
	 			\left[
	 			\begin{matrix}
	 				E_1 
	 				\\
	 				E_2
	 				\\
	 				E_3
	 				\\
	 				E_4
	 				\\E_5
	 			\end{matrix}
	 			\right]
	 		\end{equation}
	 		where $\mu_i$ are real-valued 1 forms, and $c_i^j$ are complex functions satisfying $c_{i}^j=- \overline{c_{j}^i}$. 
	 		
	 		Denote the column matrix $\{e_1, e_2, e_3, e_4, e_5\}$ by $e$ and  $\{E_1, E_2, E_3, E_4, E_5\}$ by $E.$ We have 
	 	$e=A E$, where $A$ is the $\U(5)$-valued function defined by \eqref{Relatng e and E}. Differentiate and apply the definitions of $W$ and $\tilde W$ to obtain
	 		\begin{equation}\label{dA}
	 			dA= WA-A\tilde{W}.
	 		\end{equation}
	 		By comparing the first entry of the first row of equation \eqref{dA}, 
	 		and taking conjugate, we have 
	 		\begin{equation*}
	 			d\overline{u}= \overline{u}\sqrt{-1}\left(\mu_1-\nu_1\right) \mod dz. 
	 		\end{equation*}
	 		Similarly, by comparing the last entry of the first row, and  taking the conjugate, we get
	 		\begin{equation*}
	 			d\overline{v}= \overline{v}\sqrt{-1}\left(\mu_5-\nu_1\right) \mod dz. 
	 		\end{equation*}
	 		Therefore, from   \cite[Lemma 2.1]{Li99}, $\left|u\right|$ and $\left|v\right|$  are  absolute value type functions.
	 		Together with \eqref{MC froms}, \eqref{Relatng e and E} and \eqref{MC froms of E}, 
	 		equation \eqref{dA} also gives:
	 		\begin{equation}\label{relationabc}
	 			\begin{aligned}
	 				 a_{1}^{2}=u c_{1}^{2}, \quad b_{1}^{4}=- v \overline{c_{4}^{5}}, \quad a_{2}^{3}=c_{2}^{3},\quad  b_{2}^5=v \overline{c_{1}^{2}}, \quad a_{3}^{4}= c_{3}^{4},\quad a_{4}^{5}= uc_{4}^{5}
	 			\end{aligned}.
	 		\end{equation}
	 		Since $|u|$ and $|v|$ are absolute value type functions, then by \cite[Lemma 2.1]{Li99}, together with the Maurer-Cartan structure equations for $W$ and $\tilde W$, 
	 		\begin{equation}\label{laplacian of u}
	 			\begin{aligned}
	 				\frac{1}{4}\Delta_{S^2}\log\left|u\right|^2 \phi\wedge\overline{\phi}&= \sqrt{-1}d\left(\mu_1-\nu_1\right)=d\tilde{w}_{1}^{1}-d \omega_{1}^{1}\\&= \tilde{w}_{1}^{2}\wedge \tilde{w}^{1}_{2} -\left(\omega_{1}^{2}\wedge \omega_{2}^{1}+\omega_{1}^{4}\wedge \omega_{4}^{1}+\omega_{1}^{5}\wedge \omega_{5}^{1}\right)\\&=-\left|c_{1}^{2}\right|^2\phi\wedge\overline{\phi} +\left|a_{1}^{2}\right|^2\phi\wedge\overline{\phi}+\left|b_{1}^{4}\right|^2\overline{\phi}\wedge{\phi} +\left|b_{1}^{5}\right|^2\overline{\phi}\wedge{\phi}\\&= \left(-\left|c_{1}^{2}\right|^2+\left|a_{1}^{2}\right|^2-\left|b_{1}^{4}\right|^2 -\left|b_{1}^{5}\right|^2\right)\phi\wedge\overline{\phi}.
	 			\end{aligned}
	 		\end{equation}		
	 		By using \eqref{normalize components}, we have
	 		\begin{equation}\label{b2b2}	-\left|b_{1}^{4}\right|^2 -\left|b_{1}^{5}\right|^2	=	-\frac{1}{\lambda_0}\left(1-\lambda_1|a_2^3|^2-\lambda_0|b_2^5|^2-\lambda_2|a_3^4|^2 \right).
	 		\end{equation}	
	 		Plugging this into 	\eqref{laplacian of u}, using  \eqref{relationabc} and the normalization $|u|^2+|v|^2=1$, we obtain
	 		$$\frac{1}{4}\Delta_{S^2}\log\left|u\right|^2 =\frac{1}{\lambda_0}\left(-1+\lambda_{1}\left|c_{2}^{3}\right|^2+\lambda_{2}\left|c_{3}^{4}\right|^2
	 		\right).$$
	 		Similarly, we compute 
	 		\begin{equation}\label{laplacian of va}
	 			\begin{aligned}
	 				\frac{1}{4}\Delta_{S^2}\log\left|v\right|^2 \phi\wedge\overline{\phi}&= \sqrt{-1}d\left(\mu_5-\nu_1\right)=d\tilde{w}_{5}^{5}-d \omega_{1}^{1}\\&= \tilde{w}_{5}^{4}\wedge \tilde{w}^{5}_{4} -\left(\omega_{1}^{2}\wedge \omega_{2}^{1}+\omega_{1}^{4}\wedge \omega_{4}^{1}+\omega_{1}^{5}\wedge \omega_{5}^{1}\right)\\&=-\left|c_{4}^{5}\right|^2\overline{\phi}\wedge {\phi} +\left|a_{1}^{2}\right|^2\phi\wedge\overline{\phi}+\left|b_{1}^{4}\right|^2\overline{\phi}\wedge{\phi} +\left|b_{1}^{5}\right|^2\overline{\phi}\wedge{\phi}\\&= \left(\left|c_{4}^{5}\right|^2+\left|a_{1}^{2}\right|^2-\left|b_{1}^{4}\right|^2 -\left|b_{1}^{5}\right|^2\right)\phi\wedge\overline{\phi}.
	 			\end{aligned}
	 		\end{equation}
	 		Plugging \eqref{b2b2} into 	\eqref{laplacian of va}, using  \eqref{relationabc} and the normalization $|u|^2+|v|^2=1$, we obtain
	 		\begin{equation}\label{laplacian of v}
	 			\frac{1}{4}\Delta_{S^2}\log\left|v\right|^2 =\left|c_{1}^{2}\right|^2+ \left|c_{4}^{5}\right|^2+\frac{1}{\lambda_0}\left(-1+ \lambda_{1}\left|c_{2}^{3}\right|^2+\lambda_{2}\left|c_{3}^{4}\right|^2\right).
	 		\end{equation} 
	 		
	 		Now, let $\tilde \sigma_i$ be a nonzero holomorphic local section of the $i$-th osculating curve of $g_0:S^2\to \CP^4$, and set $\xi_i= \|\tilde \sigma_i\|^2$. In view of \eqref{metricholocp}, we have, for  $1\le i \le 4$,
	 		$\left|c_{i}^{i+1}\right|^2\phi\overline \phi=\partial_z\partial_{\bar z}\log\xi_{i-1}dz d\bar z.$ 
	 	 Since
	 		$\Delta_{S^2}=\frac{4}{\eta^2}\partial_z\partial_{\bar z}$, 	where $\eta^2$ is the conformal factor defined by \eqref{flag metric}, we get
	 		\begin{equation}\label{ci}
	 			\left|c_{i}^{i+1}\right|^2=\frac14 \Delta_{S^2}\log \xi_{i-1}.
	 		\end{equation}
	 		It follows from \eqref{laplacian of u}, \eqref{laplacian of v} and \eqref{ci} that
	 		\begin{equation}\label{lap}
	 			\begin{aligned}
	 				\frac{1}{4}\Delta_{S^2}\log\frac{\left|u\right|^2}{\xi_{1}^{{\lambda_{1}}/{\lambda_0}}\xi_{2}^{{\lambda_{2}}/{\lambda_0}}}&=-\frac{1}{\lambda_0},\quad 
	 				\frac{1}{4}\Delta_{S^2}\log\frac{\left|v\right|^2}{\xi_{0}\xi_3\xi_{1}^{\lambda_1/\lambda_0}\xi_{2}^{\lambda_{2}/\lambda_0}}&=-\frac{1}{\lambda_0}	\end{aligned}.
	 		\end{equation}
	 		As quotients of generalized absolute value type functions on $S^2$, we know that
	 		$$
	 		\frac{\left|u\right|^2}{\xi_{1}^{{\lambda_{1}}/{\lambda_0}}\xi_{2}^{{\lambda_{2}}/{\lambda_0}}},\quad \frac{\left|v\right|^2}{\xi_{0}\xi_3\xi_{1}^{\lambda_1/\lambda_0}\xi_{2}^{\lambda_{2}/\lambda_0}}$$
	 		are generalized absolute value type functions on $S^2$. Then it follows from   Lemma \ref{Absolute function lemma2} together with \eqref{lap} that
	 		\begin{equation*}
	 				\mathrm{order}\left(\frac{\left|u\right|^2}{\xi_{1}^{{\lambda_{1}}/{\lambda_0}}\xi_{2}^{{\lambda_{2}}/{\lambda_0}}}\right)=\frac{4}{K\lambda_0},\quad  \mathrm{order}\left(\frac{\left|v\right|^2}{\xi_{0}\xi_3\xi_{1}^{\lambda_1/\lambda_0}\xi_{2}^{\lambda_{2}/\lambda_0}}\right)=\frac{4}{K\lambda_0},
	 		\end{equation*}
	 		where $K$ is the constant curvature of $\Psi$;
	 		moreover,
	 		\begin{equation}\label{u+v=1}
	 			\begin{aligned}
	 				\left|u\right|^2&=\frac{\left|g(z)\right|^2}{\left(1+z\bar{z}\right)^{\frac{4}{K\lambda_0}}}\xi_{1}^{{\lambda_{1}}/{\lambda_0}}\xi_{2}^{{\lambda_{2}}/{\lambda_0}},\quad  \left|v\right|^2 &=\frac{\left|h(z)\right|^2}{(1+z\bar{z})^\frac{4}{K\lambda_0}}
	 				\xi_{0}\xi_3\xi_{1}^{\lambda_1/\lambda_0}\xi_{2}^{\lambda_{2}/\lambda_0}\end{aligned}
	 		\end{equation}
	 		where
	 		\begin{equation*}\label{g(z) h(z)}
	 			\begin{aligned}
	 				 \left|g(z)\right|^2=c^2 \frac{\left|z-z_1\right|^{2p_1}\ldots \left|z-z_m\right|^{2p_m}}{\left|z-z_{m+1}\right|^{2q_1}\ldots \left|z-z_{m+k}\right|^{2q_k}},\quad  \left|h(z)\right|^2=d^2 \frac{\left|z-w_1\right|^{2r_1}\ldots \left|z-w_n\right|^{2r_n}}{\left|z-w_{n+1}\right|^{2s_1}\ldots \left|z-w_{n+j}\right|^{2s_j}}
	 			\end{aligned},
	 		\end{equation*} 
	 		with $c,d\in\mathbb{R}$. 

	 		Since $\left|u\right|^2+\left|v\right|^2=1$, equation \eqref{u+v=1} gives
	 		\begin{equation}\label{equation in beta}
	 			(1+z\bar{z})^\frac{4}{K\lambda_0}=\xi_{1}^{{\lambda_{1}}/{\lambda_0}}\xi_{2}^{{\lambda_{2}}/{\lambda_0}}\left(\left|g(z)\right|^2+ \left|h(z)\right|^2 \xi_{0} \xi_{3}\right).
	 		\end{equation}
	 			Since all (nonconstant) polynomials $\xi_j(z,\bar z)$, with $j\in\{0,\ldots,p-1\}$, and the factor $1+z\bar z$ are nonvanishing, the functions
	 		$z\mapsto \xi_j^{\alpha}(z,\bar z)$ and $z\mapsto  (1+z\bar z)^\alpha$
	 		are real analytic on $\mathbb{C}$, even when  $\alpha$ is not an integer. The corresponding complexifications are given, respectively, by 
	 		$(z,w)\mapsto \xi_j^{\alpha}(z,w)$ and $(z,w)\mapsto  (1+zw)^\alpha,$
	 		considering the principal branch of each multivalued function $Z\mapsto Z^\alpha$. 
	 		Now, take the complexifications of both sides of  \eqref{equation in beta};
	 		by uniqueness of complexification,
	 		we have 	

	 			\begin{equation*}\label{equation in betazw}
	 				\left(1+zw \right)^{\frac{4}{K\lambda_0}}=
	 				\xi_{1}^{\lambda_{1}/\lambda_0}(z,w)\xi_{2}^{\lambda_{2}/\lambda_0}(z,w)
	 				\left(g(z)\overline g(w) +
	 			h(z)\overline h(w) \xi_{0}(z,w) \xi_{3}(z,w)\right).
	 		\end{equation*}
	 		Hence, if not empty, the zero set of the polynomial $\xi_i(z,w)$, with $i\in\{1,2\}$, coincides in $\C^2$ with the zero set of the irreducible polynomial $1+zw$.  Thus, for some positive integers $\alpha_1, \alpha_2$ and positive real numbers $c_1,c_2$,  we have
	 		$\xi_1(z,\bar z) =c_1\left(1+z\bar{z}\right)^{\alpha_{1}}$, and $\xi_2(z,\bar z) =c_2\left(1+z\bar{z}\right)^{\alpha_{2}}.$
	 		In view of  \eqref{metricgamma} and \eqref{curvaturepsi}, this indicates that the harmonic map $g_2:S^2\to \CP^4$ is an immersion of  constant curvature, hence, by \cite[Theorem 5.4]{Bolton}, the harmonic sequence $g_0,g_1,g_2,g_3,g_4$ is the $4$-Veronese sequence, up to isometry. 
	 		Consequently, all the $\xi_{i}(z,\bar z)$, with $i=0,1,2,3$, are of the form $c_i(1+z\bar z)^{\alpha_i}$. Thus,  equation \eqref{equation in beta}  becomes
	 		\begin{equation}\label{equation in z zbar}
	 			\left(1+z\bar{z}\right)^{\alpha}=C_1\left(\left|g(z)\right|^2+C_2\left|h(z)\right|^2\left(1+z\bar{z}\right)^{\alpha'}\right),
	 		\end{equation}
	 		where  $\alpha=\frac{4}{K\lambda_0}-\frac{\lambda_1\alpha_1+\lambda_2\alpha_2}{\lambda_0}$, $\alpha'=\alpha_0+\alpha_3$ and $C_1,C_2$ are positive real numbers.

	 		Now, suppose that $\alpha > 0$. Complexifying 
	 		both sides of  \eqref{equation in z zbar}, we obtain an 
	 		equality between two analytic functions in the independent complex variables $z,w$. The left-hand side will vanish whenever $w=-\frac{1}{z}$;  then, in view of the general form of $|g(z)|^2$, equality
	 		\eqref{equation in z zbar} can occur only if
	 		$g(z)=0$. This implies that $u=0$, and we are done. 
	 		
	 		Suppose that $\alpha \leq 0$. From \eqref{equation in z zbar} we obtain
	 		$$1=C_1\big(\left|g(z)\right|^2(1+z\bar z)^{-\alpha}+C_2\left|h(z)\right|^2\left(1+z\bar{z}\right)^{\alpha'-\alpha}\big).$$
	 		This implies that $\left|h(z)\right|^2\left(1+z\bar{z}\right)^{\alpha'-\alpha}$ is bounded, which occurs if, and only if, $h(z)=0$ (hence $v=0$), because $\alpha'-\alpha>0$. 
	 		Then diagram \eqref{refinement of the diagram} becomes
	 		\begin{equation*}\label{refinement of the diagram1}
	 			\begin{tikzpicture}
	 				\node (g4) at (0, 1.1) {$g_4$};
	 				\node (g1) at (1.1, 1.1) {$g_1$};
	 				\node (g2) at (2.2, 1.1) {$g_2$};
	 				\node (g3) at (0, 0) {$g_3$};
	 				\node (g0) at (1.1, 0) {$g_0$};
	 				\draw[->] (g0) -- (g1);
	 				\draw[->] (g1) -- (g2);
	 			
	 				\draw[->] (g3) -- (g4);
	 			
	 				\draw[->, >=stealth,] (g2) to[out=250, in=-15, looseness=1.25] (g3);
	 			\end{tikzpicture}.
	 		\end{equation*}
	 		This leads to a contradiction with the assumption that $\psi_0, \psi_{1}, \psi_{2}$ is the harmonic sequence generated by $\psi_{0}$; in fact, if this last diagram was valid, then {$\partial'$-Gauss bundle} of $\psi_{0}=g_3\oplus g_4$ would be zero. 
	 	\end{proof}

	 	\begin{remark}\label{not harmonic sequence}
	 	A careful inspection of the proofs of Theorem \ref{mainresult2} and Theorem \ref{mainresult3} shows that, besides \eqref{primitivemainresult} and \eqref{isotropictheorem}, any other primitive immersion $\Psi:S^2\to F_{2,2,1}$ of constant curvature with respect to at least one $\U(5)$-invariant metric on $F_{2,2,1}$ can be obtained by operating with homogeneous projections,  adding constants, adding primitive maps,  and shifts on primitive immersions  of constant curvature into other flag manifolds ($F_{1,1}=\CP^1$, $F_{2,1}$, $F_{1,1,1}$, $F_{2,2}=G_2(\C^4)$, $F_{2,1,1}$, $F_{1,1,1,1}$, $F_{1,1,1,1,1}$)  which are primitive lifts of  linearly full harmonic maps. For example, \eqref{lotsV} corresponds to the primitive immersion $\Psi_1:S^2\to  F_{2,2,1}$ of constant curvature given by
	 	$\Psi_1=\left(V_0^4\oplus V_1^4,V_2^4\oplus V_3^4,V_4^4\right).$ This is not a primitive lift of $\psi_0=V_0^4\oplus V_1^4$, in the sense that $V_2^4\oplus V_3^4$ is not the first $\partial'$-Gauss bundle of $\psi_0$.
	 	However, we have $\Psi_1=\Pi\circ \tilde\Psi_1$ with  $ \tilde\Psi_1:S^2\to F_{1,1,1,1,1}$
	 	given by
	 	$\tilde \Psi_1=\left(V_0^4, V_1^4,V_2^4, V_3^4,V_4^4\right).$
	 	For another example, consider $f_0$ constant in \eqref{diag:G32a}. This case corresponds to  the primitive immersion $\Psi_2:S^2\to  F_{2,2,1}$ of constant curvature given by
	 	$\Psi_2=\left(V_0^2\oplus \C,V_1^2\oplus \C,V_2^2\right),$
	 	which is obtained  from $\tilde \Psi_2:S^2\to F_{1,1,1}$ given by $\tilde \Psi_2=(V_0^2,V_1^2,V_2^2)$
	 	by adding constants.
	 \end{remark}

\end{document}